\documentclass[a4paper,reqno,11pt]{amsart} 

\usepackage[T1]{fontenc}
\usepackage[utf8]{inputenc}
\usepackage{xspace,
		amsfonts}
\usepackage{amsmath} 
\usepackage{amssymb}
\usepackage{graphicx,
		bbm,
		tikz,
		subfig}
  \usepackage{xfrac}
		\usepackage{dsfont}
\usepackage{amsthm}
\usepackage{mathtools}
\usepackage{enumitem}
\usepackage{verbatim}
\usepackage[autostyle]{csquotes}

\usepackage{cite}

\mathtoolsset{showonlyrefs}

\usepackage{geometry}
\newtheorem{theorem}{Theorem}[section]
\newtheorem{lemma}[theorem]{Lemma}
\newtheorem{proposition}[theorem]{Proposition}

\newtheorem{definition}[theorem]{Definition}

\newtheorem{remark}[theorem]{Remark}

\newcommand{\deb}{\rightharpoonup}

\newcommand{\R}{\mathbb{R}}
\newcommand{\Rp}{\mathbb{R^{+}}}
\newcommand{\G}{\mathcal{G}}

\newcommand{\N}{\mathbb{N}}
\newcommand{\E}{\mathbb{E}}

\newcommand{\eps}{\varepsilon}
\newcommand{\ep}{\varepsilon}
\newcommand{\V}{\mathbb{V}}

\newcommand{\HH}{\mathcal{H}}
\newcommand{\K}{\mathcal{K}}

\newcommand{\f}{\frac}

\newcommand{\ee}{\textit{e}}
\newcommand{\D}{\mathcal{E}}
\newcommand{\la}{\lambda}

\newcommand{\Z}{\mathbb{Z}}
\newcommand{\Q}{\mathcal{Q}}

\newcommand{\EE}{\mathcal{E}}
\newcommand{\dx}{\,dx}

\newcommand\vv{\textsc{v}}
\newcommand\ww{\textsc{w}}

\tikzstyle{nodino}=[circle,draw,fill,inner sep=0pt,minimum size=0.5mm]
\tikzstyle{infinito}=[circle,inner sep=0pt,minimum size=0mm]
\tikzstyle{nodo}=[circle,draw,fill,inner sep=0pt, minimum size=0.5*width("k")]
\tikzstyle{nodo_vuoto}=[circle,draw,inner sep=0pt, minimum size=0.5*width("k")]
\usetikzlibrary{graphs}
\tikzset{every loop/.style={min distance=10mm,in=300,out=240,looseness=10}}
\tikzset{place/.style={circle,thick,draw=blue!75,fill=blue!20,minimum
		size=6mm}}
\tikzset{place2/.style={circle,thick,draw=red!75,fill=red!20,minimum
		size=6mm}}

\title[ ]{Normalized ground states for Schr\"odinger equations on metric graphs with nonlinear point defects }
\author[ ]{Filippo Boni}
\address[F. Boni]{Scuola Superiore Meridionale, Largo S. Marcellino, 10, 80138, Napoli, Italy.}
\email{f.boni@ssmeridionale.it}

\author[ ]{Simone Dovetta}
\address[S. Dovetta]{Politecnico di Torino, Dipartimento di Scienze Matematiche "G.L. Lagrange", Corso Duca degli Abruzzi 24, 10129, Torino, Italy.}
\email{simone.dovetta@polito.it}

\author[ ]{Enrico Serra}
\address[E. Serra]{Politecnico di Torino, Dipartimento di Scienze Matematiche "G.L. Lagrange", Corso Duca degli Abruzzi 24, 10129, Torino, Italy.}
\email{enrico.serra@polito.it}

\begin{document}

\begin{abstract}
	We investigate the existence of normalized ground states for Schr\"odinger equations on noncompact metric graphs in presence of nonlinear point defects, described by nonlinear $\delta$-interactions at some of the vertices of the graph. For graphs with finitely many vertices, we show that ground states exist for every mass and every $L^2$-subcritical power. For graphs with infinitely many vertices, we focus on periodic graphs and, in particular, on $\Z$-periodic graphs and on a prototypical $\Z^2$-periodic graph, the two--dimensional square grid. We provide a set of results unravelling nontrivial threshold phenomena both on the mass and on the nonlinearity power, showing the strong dependence of the ground state problem on the interplay between the degree of periodicity of the graph, the total number of point defects and their dislocation in the graph.
\end{abstract}

\maketitle


\section{Introduction}

In this paper we analyze existence and nonexistence of normalized solutions to the Schr\"odinger equation on metric graphs in the presence of   {\em nonlinear point defects} at some of the vertices. Precisely, given a connected metric graph $\G=(\V_\G,\E_\G)$ and a subset $V \subseteq \V_\G$ of its vertices, for fixed $\mu >0$ we study the existence of a continuous $u:\G \to \R$ satisfying
\begin{equation}
\label{prob}
\begin{cases}
u'' = \lambda u	&\text{ on every edge of $\G$},\\
\|u\|_{L^2(\G)}^2 = \mu & \\
\displaystyle\sum_{\ee\succ \vv}u_{e}'(\vv)=-|u(\vv)|^{q-2}u(\vv)&  \text{ at every } \vv\in V,\\
\displaystyle\sum_{e\succ\vv} u_e'(\vv)=0	&\text{ at  every } \vv \in \V_\G \setminus V\\
\end{cases}
\end{equation}
for some $\lambda \in \R$.
In problem \eqref{prob}, the exponent $q$ satisfies $q \in (2,4)$ and the symbol $e\succ\vv$ means that the sum is extended to all edges incident at the vertex $\vv$.

The last requirement in \eqref{prob} is the Kirchhoff, or natural, boundary condition, while 
\[
\displaystyle\sum_{\ee\succ \vv}u_{e}'(\vv)=-|u(\vv)|^{q-2}u(\vv), \qquad \forall\, \vv \in V,
\]
can be thought of as representing the effect of
a deep attractive potential well or also a strong attractive defect at all vertices $\vv \in V$. In the literature it is customary to call such 
condition a {\em nonlinear $\delta$--interaction} and interpret it as a model for strongly localized, point-like defects or inhomogeneities in the medium that supports the  propagation (see e.g. \cite{ABR} for a wide overview on non-Kirchhoff vertex conditions on graphs). Models with concentrated nonlinearities have been proposed e.g. in semiconductor theory \cite{JPS, N} to describe the quantum dynamics in resonant tunneling diodes, as well as the effect of the confinement of charges in small regions. 

Solutions to \eqref{prob} correspond, {\em formally}, to standing waves for the nonlinear Schr\"odinger equation
\[
i\partial_t \psi +\partial_{xx} \psi + \sum_{\vv\in V} \delta_{\vv} |\psi|^{q-2}\psi = 0\qquad \text{ on } \G
\]
via the ansatz $\psi(t,x) = e^{i\lambda t}u(x)$. Finally, the constraint on the $L^2$ norm of $u$ in \eqref{prob}, which accounts for the specification {\em normalized} attached to the solutions of \eqref{prob}, is a standard requirement on the {\em mass} of $u$ (sometimes also interpreted as the number of particles in the condensate), a quantity that is preserved along the evolution in time.

The solutions of \eqref{prob} can be found variationally as the critical points of the energy functional 
\[
E_{q,V}(u,\G) := \frac12 \|u'\|_{L^{2}(\G)}^{2}-\frac1q\sum_{\vv\in V}|u(\vv)|^q
\]
on the space of mass-constrained functions
\[
H^{1}_{\mu}(\G):=\left\{w\in H^{1}(\G) \mid \|w\|_{L^{2}(\G)}^{2}=\mu\right\}
\]
(for standard definitions of Lebesgue and Sobolev spaces on metric graphs see e.g. \cite{AST}).
In this respect, the condition $q \in (2,4)$ corresponds to the so-called $L^{2}$-subcritical case, i.e. the case in which the energy functional $E_{q,V}$ is bounded from below in $H^{1}_{\mu}(\G)$ for every value of the mass $\mu$ and every set $V$ (see Lemma \ref{lem:level} below). 

Due to their possible relevance in the applications, in this paper we are only concerned with {\em ground states}, namely with functions $u \in H^1_\mu(\G)$ solving the problem
\[
E_{q,V}(u,\G) = \inf_{w \in H^1_\mu(\G)} E_{q,V}(w,\G) =: \EE_{q,V}(\mu,\G).
\]

\smallskip
\noindent Although the study of nonlinear Schr\"odinger equations on metric graphs is relatively recent, there is by now a very rich and steadily increasing literature, mostly devoted to the search of ground states or bound states in presence of the diffused, standard nonlinearity $\int_\G |u|^p\dx$, under various conditions on $p$, both for the energy functional (see for instance \cite{acfn_jde,acfn_aihp,ADST,AST,AST1,AST2,BMP,BDL20,BDL21,CJS,D-per,DT22,KMPX,KNP,NP,PS20,PSV}) and for the action functional (where $\lambda$ is fixed and the mass is unknown, see e.g. \cite{ACT,DDGS,DDGS2,pankov}). A certain attention has been devoted also to localized nonlinearities in the form $\int_\K|u|^p\,dx$, that is when the nonlinear term is still some $L^p$ norm of the function but restricted to a prescribed subset $\K$ of the graph $\G$ (see \cite{DT,LZL,ST-JDE,ST-NA,T-JMAA}). 

Conversely, the analysis of nonlinear point defects on graphs, i.e. nonlinear terms as in $E_{q,V}$ above, seems to be at its very beginning and, up to our knowledge, the available results concern existence and nonexistence of normalized ground states for models involving both standard and pointwise nonlinearities on noncompact graphs with finitely many vertices and edges (see \cite{ABD,BC,BD,BD22}, and also \cite{acfn_jde,acfn_aihp} for the case of linear point defects on star graphs). These first investigations for the doubly nonlinear model proved that the presence of pointwise nonlinearities is a source of new phenomena, sensibly different from the ones usually observed with standard nonlinearities only. 

\begin{figure}[t]
	\centering
	\includegraphics[width=0.5\textwidth]{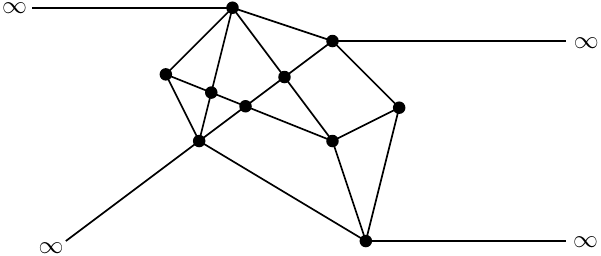}
	\caption{A noncompact graph in $\mathbf{G}$ with finitely many vertices and edges.}
	\label{fig:rette}
\end{figure}

\smallskip
The purpose of the present work is to investigate normalized ground states for the model $E_{q,V}$ with the sole pointwise nonlinearity. Our aim is to work at a high level of generality, enquiring how the presence of point defects at some of the vertices affects the existence of ground states on some classes of noncompact graphs. In particular, for reasons that will be clear once we state our main results, our interest will be mainly devoted to noncompact {\em periodic graphs}, i.e. graphs with infinitely many vertices and edges arranged in some periodic pattern. However, for the sake of completeness, we will also provide some results for noncompact graphs with finitely many vertices and edges.

To simplify the notation, in what follows we denote by $\mathbf{G}$ the class of graphs $\G = (\V_\G, \E_\G)$ such that
\begin{itemize}
\item $\G$ is connected and has an at most countable number of edges;
\item $\deg(\vv) < \infty$ for every $\vv \in \mathbb{V}_\G$, where $\deg(\vv)$ denotes the degree of the vertex $\vv$, i.e.\ the  number of edges incident at $\vv$;
\item $\displaystyle\inf_{e \in {\mathbb E}_\G} |e| >0$, where $|e|$ denotes the length of the edge $e$.
\end{itemize}
All the graphs considered in this paper will be noncompact. For graphs in the class $\bf G$ noncompactness is ascribable either to the presence of some unbounded edges, as usual identified with a half-line, or to the set $\E_\G$ being infinite.  

It is reasonable to expect the existence of ground states of $E_{q,V}$ on a graph $\G$ to depend not only on topological or metric properties of the graph itself, but also on the total number of point defects (i.e. the cardinality of the set $V$) and perhaps on how they are dislocated in the structure. Actually, with our first result we show that, if the graph has finitely many vertices, the problem is insensitive to any other features of $\G$ and $V$ and ground states always exist.
\begin{theorem}
\label{thm:n-cpt-hl}
Let $\G\in \mathbf{G}$ be noncompact with $\#\V_\G < \infty$. Then, for every nonempty $V\subseteq \V_\G$, every $q\in(2,4)$ and every $\mu>0$, there results $\EE_{q,V}(\mu,\G)<0$ and ground states of $E_{q,V}$ in $H_\mu^1(\G)$ exist. 
\end{theorem}
Observe that, if $\G\in\mathbf{G}$ is noncompact and such that $\#\V_\G<\infty$, then $\G$ has finitely many edges and at least one is a half-line (see Figure \ref{fig:rette}). For this type of graphs, the fact that ground states with point defects always exist already marks a difference with models involving standard nonlinearities, for which it is by now well-known that there are both topological and metric conditions ruling out existence of ground states (see e.g. \cite{AST}).

Let us then consider graphs in $\mathbf{G}$ with infinitely many vertices, i.e. $\#\V_\G=\infty$. Since this class of graphs is extremely large and contains objects sensibly different from each other, here we will not  treat it in its full generality. In fact, also in view of their possible relevance, we focus on periodic graphs, i.e. graphs with infinitely many vertices and edges arranged in a given periodic fashion. 

We avoid reporting a rigorous and general definition of periodic graph, for which we refer the interested reader to \cite[Definition 4.1.1]{BK}, and we rather focus in detail on two subclasses of periodic graphs. First, we will consider $\Z$-periodic graphs, for us being graphs obtained gluing together in a $\Z$-symmetric pattern infinitely many copies of a given compact graph (see Figure \ref{fig:Zper} for some examples and Section \ref{sec:prel} for the precise definition). Second, we will focus on the two-dimensional square grid (see Figure \ref{fig:Q}), which we take as a prototypical model for $\Z^2$-periodic graphs. 

\begin{figure}[t]
	\centering
	\subfloat{\includegraphics[width=0.4\columnwidth]{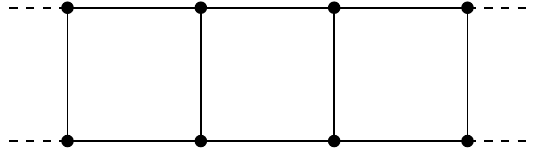}}\qquad\quad
	\subfloat{\includegraphics[width=0.45\columnwidth]{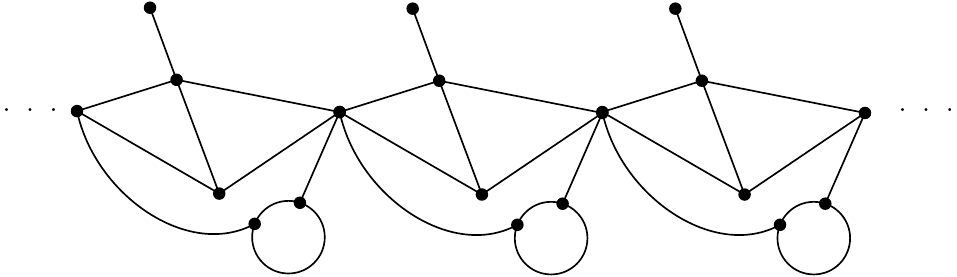}}
	\caption{Examples of $\Z$-periodic graphs.}
	\label{fig:Zper}
\end{figure}

As for $\Z$-periodic graphs, our main result is the following.
\begin{theorem}
\label{thm:Z-per}
Let $\G\in\mathbf{G}$ be a $\Z$-periodic graph. Then, for every nonempty $V\subseteq\V_\G$, every $q\in(2,4)$ and every $\mu>0$, there results $\EE_{q,V}(\mu,\G)<0$. Moreover, 
\begin{itemize}
\item[(i)] if $\#V<+\infty$,  ground states of $E_{q,V}$ in $H_\mu^1(\G)$ exist for every $q\in (2,4)$ and every $\mu>0$;
\item[(ii)] if $\#V=+\infty$ and $V$ is a $\Z$-periodic subset of $\V_\G$, then ground states of $E_{q,V}$ in $H_\mu^1(\G)$ exist for every $q\in (2,4)$ and every $\mu>0$;
\item[(iii)] there exists $V$ (with $\#V=+\infty$), for which ground states of $E_{q,V}$ in $H_\mu^1(\G)$ do not exist for any $q\in(2,4)$ and any $\mu>0$.
\end{itemize}
\end{theorem}
Theorems \ref{thm:n-cpt-hl}--\ref{thm:Z-per} establish a similarity between graphs with finitely many vertices and $\Z$-periodic graphs, since in both classes of graphs the ground state level $\EE_{q,V}$ is always strictly negative and ground states always exist when the number of point defects is finite. However, since on $\Z$-periodic graphs it is possible to have infinitely many point defects, the existence of ground states may be affected also by a loss of compactness at infinity. On the one hand, Theorem \ref{thm:Z-per}(ii) shows that, when the set of point defects $V$ is infinite  but is itself $\Z$-periodic (i.e. it has the same periodicity of the graph, see Section \ref{sec:prel} below for a precise definition), one recovers enough compactness to ensure existence of ground states. On the other hand, even though at a first glance the periodicity of $V$ may seem a quite restrictive assumption, Theorem \ref{thm:Z-per}(iii) proves that, dropping it, one can easily exhibit infinite sets $V$  for which existence of ground states is ruled out, independently of $q$ and $\mu$.

\begin{figure}[]
	\centering
	\begin{tikzpicture}[xscale= 0.5,yscale=0.5]
	\draw[step=2,thin] (0,0) grid (8,8);
	\foreach \x in {0,2,...,8} \foreach \y in {0,2,...,8} \node at (\x,\y) [nodo] {};
	\foreach \x in {0,2,...,8}
	{\draw[dashed,thin] (\x,8.2)--(\x,9.2) (\x,-0.2)--(\x,-1.2) (-1.2,\x)--(-0.2,\x)  (8.2,\x)--(9.2,\x); }
	\end{tikzpicture}
	\caption{The infinite two--dimensional square grid $\Q$.}
	\label{fig:Q}
\end{figure}
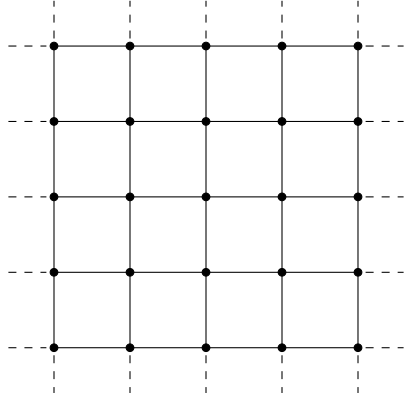

Let us now turn to the two--dimensional square grid $\Q$ with edges of unitary length (Figure \ref{fig:Q}). This is, in our opinion, the most interesting case discussed in this paper, since new phenomena occur, revealing  a strong dependence of the ground state problem not only on the dislocation of point defects, but also on their total number, in contrast to the other graphs considered so far. 

We begin by assuming that $V$ is finite. In this case, we have the next general result.

\begin{theorem}
\label{thm:Vfin1}
Let $\Q$ be the two--dimensional square grid. Then, for every nonempty $V\subset\V_\Q$ with $\#V<+\infty$ and every $q\in(2,4)$, there exists a critical mass $\mu_q^*>0$, depending on $V$ and $q$, such that
\begin{equation*}
\EE_{q,V}(\mu,\Q)\begin{cases}
=0 & \text{if }\mu\leq\mu_q^*\\
<0 & \text{if }\mu>\mu_q^*\,.
\end{cases}
\end{equation*} 
Moreover, 
\begin{itemize}
\item[(i)] if $\mu<\mu_q^*$,  ground states of $E_{q,V}$ in $H_\mu^1(\Q)$ do not exist;
\item[(ii)] if $\mu>\mu_q^*$,  ground states of $E_{q,V}$ in $H_\mu^1(\Q)$ exist.
\end{itemize}
Furthemore,
\[
\lim_{q\to2^+}\mu_q^*=0\,.
\]
\end{theorem}

This theorem describes a general, abstract feature of the energy $E_{q,V}$ on the grid: when the number of point defects is finite, independently of any other properties of $V$, the ground state level undergoes a sharp transition from 0 to strictly negative values as soon as the mass crosses a threshold $\mu_q^*$. This is sensibly different with respect to all graphs discussed before and suggests that  in presence of $\Z^2$-periodicity, finitely many nonlinear point defects are not strong enough to trap ground states at small masses. It is also worth noticing that this is a purely nonlinear effect, since it is easy to prove that ground states with linear point defects (i.e. the energy above with $q=2$) always exist on $\Q$ (see Lemma \ref{lem:bot-spec<0}). This is further confirmed by the asymptotic behaviour of the critical mass $\mu_q^*$ as $q$ approaches $2$ obtained in Theorem \ref{thm:Vfin1}. 

In the special case of $V$ containing a single vertex, it is also possible to derive qualitative properties of ground states.
\begin{proposition}
\label{prop:rad}
Let $\Q$ be the two--dimensional square grid and $V=\{\overline\vv\}$, for some given $\overline{\vv}\in\V_\Q$. Let $u\in H_\mu^1(\Q)$ be a ground state of $E_{q,V}$ (which exists for every $q\in (2,4)$ if $\mu > \mu_q^*$ by the preceding result). Then $u$ is radial and radially decreasing on $\Q$ with respect to $\overline{\vv}$.
\end{proposition}

We point out that the radiality of ground states is by no means trivial. Usually, this property follows from rearrangement techniques that provide P\'olya--Szeg\H{o} type inequalities. In general, the validity of such inequalities (e.g. for functions on $\R^N$) depends in a crucial way on the isoperimetric properties of balls. Unfortunately, it turns out  that, on grids, metric balls are {\em not} isoperimetric sets. Therefore, all rearrangement inequalities based on this fundamental property simply do not hold. This is the reason why there is  no symmetry result for problems with standard nonlinearities on grids. In the case of a single point defect, however, we prove that the use of classical rearrangement inequalities can be replaced by a new type of argument, essentially based on spherical means, that allows one to obtain the result of Proposition \ref{prop:rad}.

Next we consider infinite sets $V\subset\V_\Q$. As already pointed out when dealing with $\Z$-periodic graphs, in this case it is reasonable to expect that some structural properties of $V$ are needed to obtain  ground states. To this end, we introduce the next definitions. 

\begin{definition}
\label{def:I-Z-per}
A subset $V\subset\V_\Q$ is called $\Z$-periodic if there exists a vector $\vec{v}\in \R^{2}\setminus\{0\}$ such that
\begin{itemize}
\item[$(i)$] $V=V+k\vec{v}$, for every $k\in \Z$, and
\item[$(ii)$] there exist $P_0\in \R^2$ and $r>0$ such that $|(\vv-P_0)\cdot \vec{v}^\perp|\leq r$ for every $\vv\in V$.
\end{itemize}
\end{definition}

\begin{definition}
	\label{def:I-Z2-per}
	A subset $V\subset\V_{\Q}$ is called $\Z^{2}$-periodic if there exist two linearly independent vectors $\vec{v}_1, \vec{v}_2\in \R^{2}\setminus\{0\}$ such that
	\begin{equation*}
	V=V+k_{1}\vec{v}_1+k_{2}\vec{v}_2\qquad\forall\,k_{1},k_{2}\in \Z\,.
	\end{equation*}
\end{definition}
Observe that, by definition, a $\Z$-periodic set $V$ in $\Q$ contains infinitely many vertices, but they are all contained in a strip bounded in one direction.

The next result  provides a complete characterization of the ground state problem in presence of point defects on a $\Z$-periodic set, unravelling the appearance of a rather unexpected threshold not only on the mass $\mu$, but also on the nonlinear power $q$.
\begin{theorem}
\label{thm:grid-Z-per}
Let $\Q$ be the two--dimensional square grid and $V\subset \V_\Q$ be $\Z$-periodic. Then
\begin{itemize}
\item[(i)] if $q\in(2,3)$, there results $\EE_{q,V}(\mu,\Q)<0$ and ground states of $E_{q,V}$ in $H_\mu^1(\Q)$ exist for every $\mu>0$;
\item[(ii)] if $q\in[3,4)$, there exists a critical mass $\mu_q^*>0$, depending on $q$ and $V$, such that 
\[
\EE_{q,V}(\mu,\Q)\begin{cases}
=0 & \text{if }\mu\leq\mu_q^*\\
<0 & \text{if }\mu>\mu_q^*
\end{cases}
\]
and 
\begin{itemize}
\item[(ii.1)] if $\mu<\mu_q^*$,  ground states of $E_{q,V}$ in $H_\mu^1(\Q)$ do not exist;
\item[(ii.2)] if $\mu>\mu_q^*$,  ground states of $E_{q,V}$ in $H_\mu^1(\Q)$ exist.
\end{itemize}
\end{itemize}
\end{theorem}
The situation is simpler for $\Z^2$-periodic sets of point defects.

\begin{theorem}
\label{thm:grid-Z2-per}
Let $\Q$ be the two--dimensional square grid and $V\subseteq \V_\Q$ be $\Z^2$-periodic. Then, for every $q\in(2,4)$ and every $\mu>0$, $\EE_{q,V}(\mu,\Q)<0$ and ground states of $E_{q,V}$ in $H_\mu^1(\Q)$ exist. 
\end{theorem}

Comparing Theorem \ref{thm:Z-per} with Theorems \ref{thm:Vfin1}--\ref{thm:grid-Z-per}--\ref{thm:grid-Z2-per} plainly shows that, on periodic graphs, the ground state problem we are considering is strongly sensitive to the degree of periodicity of the structure. When the degree of periodicity is the least possible (i.e. $\G$ is $\Z$-periodic), the behaviour of $\EE_{q,V}$ is rather trivial and even a single point defect is enough to make it  strictly negative. As soon as the degree of periodicity increases, on the contrary, the total number of point defects starts playing a nontrivial role. Looking at Theorems \ref{thm:Vfin1}--\ref{thm:grid-Z-per}--\ref{thm:grid-Z2-per} on the two--dimensional square grid, one sees that the main difference arises in the behaviour of $\EE_{q,V}$ at small masses. When the number of points defects is finite, $\EE_{q,V}$ is equal to 0 as soon as the mass is small enough (and ground states do not exist). If the number of point defects increases this phenomenon ceases to rule the problem, but infinitely many defects are not always enough to make $\EE_{q,V}$ strictly negative for every value of $q$ and $\mu$. On the contrary, Theorem \ref{thm:grid-Z-per} suggests that whenever the point defects are infinitely many, but constrained inside a strip which is bounded in one direction, the negativity of $\EE_{q,V}$ at small masses is recovered only for sufficiently small nonlinearity powers $q$. Strictly speaking, Theorem \ref{thm:grid-Z-per} proves this fact only for $\Z$-periodic sets $V$, but it is easy to check that the argument of the proof can be generalized to obtain the same behaviour of $\EE_{q,V}$ on sets $V$ with infinitely many vertices all contained in a strip of $\R^2$ bounded in one direction (with a threshold on $q$ possibly different than $3$). If one further expands $V$, taking e.g. $\Z^2$-periodic sets, Theorem \ref{thm:grid-Z2-per} shows that the strength of point defects is then sufficiently large to ensure the negativity of $\EE_{q,V}$ for every $q$ and $\mu$.

A heuristic explanation for this phenomenology is the following. When the mass $\mu$ is small, if the energy $E_{q,V}(u,\Q)$ of a function $u\in H_\mu^1(\Q)$ is low, it is easy to prove that $u$ is uniformly small on $\Q$. In particular, each term of the sum $\sum_{\vv\in V}|u(\vv)|^q$ is small, and so is the sum if we have few vertices $\vv$ in $V$ or large powers $q$. On the contrary, since $u$ is widespread on $\Q$, the structure of the grid forces $u$ to run through a large number of vertices, and this contributes to enlarge its kinetic energy $\|u'\|_{L^2(\Q)}^2$. Since the sign of $E_{q,V}(u,\Q)$ is the result of the competition between the kinetic energy and the total contribution of point defects, this provides a rough intuition of why finitely many defects are not enough to make $\sum_{\vv\in V}|u(\vv)|^q$ overcome the kinetic energy at small masses, as well as why this starts to occur first at small values of $q$ when $V$ contains infinitely many vertices confined in some strip of $\R^2$. 

We observe that, even though Theorems \ref{thm:Vfin1}--\ref{thm:grid-Z-per}--\ref{thm:grid-Z2-per} have been proved here only for the two--dimensional square grid $\Q$, we believe that they unravel the main features of the problem on general $\Z^2$--periodic graphs. We decided to work with the square grid because computations and proofs are particularly transparent, but we are confident that all the arguments developed here can be generalized with small effort to cover any given $\Z^2$--periodic graph. 

We point out that a nontrivial dependence of the ground state problem on the degree of periodicity of graphs has been observed also for the model with diffuse, standard nonlinearities only (see the series of works \cite{AD,ADST,D-per}). In that context, a comprehensive understanding of the problem is by now available, based on the relation between the degree of periodicity of the graph and the validity of certain Sobolev and Gagliardo--Nirenberg inequalities peculiar of higher dimensional spaces. Conversely, the general portrait for point defects on periodic graphs is far from being understood. The results of the present paper shall thus be seen as a first step in this direction and a starting point for future investigations of the problem on general $\Z^n$-periodic graphs, with $n\geq2$.

To conclude this introduction, we remark that the periodicity assumptions on  $V$ in Theorems \ref{thm:grid-Z-per}--\ref{thm:grid-Z2-per} are somewhat natural to guarantee enough compactness in the search for ground states. Analogously to what happens on $\Z$-periodic graphs, when these assumptions are dropped it is not difficult to exhibit sets of point defects on $\Q$ for which ground states never exist.
\begin{theorem}
	\label{thm:noexQ}
	Let $\Q$ be the two--dimensional square grid. There exist sets $V\subset\V_\Q$ that 
	\begin{itemize}
		\item[(i)] satisfy Definition \ref{def:I-Z-per}(ii) but do not satisfy Definition \ref{def:I-Z-per}(i); or
		\item[(ii)] satisfy Definition \ref{def:I-Z-per}(i) but do not satisfy neither Definition \ref{def:I-Z-per}(ii) nor Definition \ref{def:I-Z2-per},
	\end{itemize}
	and for which ground states of $E_{q,V}$ in $H_\mu^1(\Q)$ do not exist for any $q\in(2,4)$ and any $\mu>0.$
\end{theorem}

The remainder of the paper is organized as follows. Section \ref{sec:prel} collects some basic definitions and preliminary results. Section \ref{sec:D} describes a general framework for the study of $\EE_{q,V}$ on metric graphs. Section \ref{sec:H+Zper} deals with graphs with finitely many vertices and with $\Z$-periodic graphs, proving Theorems \ref{thm:n-cpt-hl}--\ref{thm:Z-per}. Section \ref{sec:QVfin} proves the results on the grid $\Q$ with finitely many point defects, i.e. Theorem \ref{thm:Vfin1} and Proposition \ref{prop:rad}, whereas Section \ref{sec:Vinf} discusses the case of infinitely many point defects on $\Q$ and gives the proof of Theorems \ref{thm:grid-Z-per}--\ref{thm:grid-Z2-per}--\ref{thm:noexQ}. Finally, Appendix \ref{app:linear} collects some results about linear problems related to those studied in the paper.

\subsection*{Notation} In the following, whenever possible and depending on the context we will use simplified notations as $E(u), E_q(u)$ and $\EE(\mu),\EE_{q}(\mu)$ in place of $E_{q,V}(u,\G)$ and $\EE_{q,V}(\mu,\G)$, making use of the full notation only when necessary to avoid confusion.

\section{Preliminaries}
\label{sec:prel}

We begin  recalling some preliminary definitions and estimates that will be largely used in the forthcoming sections.
\subsection{$\Z$-periodic graphs and the two--dimensional square grid $\Q$}
Since we will use it in some arguments later on, we briefly recall here the rigorous definition of $\Z$-periodic graph we adopt, taken from \cite[Section 2]{D-per}.

Let $\K\in \bf G$ be a connected compact graph, i.e. a graph with a finite number of vertices and edges, all of finite length. Let $D$ and $R$ be two non-empty subsets of $\V_\K$ and $\sigma:D\to R$ be a function such that
\begin{itemize}
    \item[$(i)$] $D\cap R=\emptyset$;
    \item[$(ii)$] $\sigma$ is bijective.
\end{itemize}
Consider then $\mathbb{G}:=\bigcup_{i\in\Z}\K_i$, where $\K_i$, $D_i$, $R_i$ are copies of $\K$, $D$, $R$ respectively, for every $i\in\Z$, and, thinking of $\sigma$ as a map from $D_i$ to $R_{i+1}$, say that two vertices $\vv,\ww$ of $\mathbb{G}$ are equivalent, writing $\vv\sim\ww$,  if either
 \begin{itemize}
     \item[$(a)$] $\vv,\ww\in \K_i$, for some $i\in\Z$, and $\vv=\ww$; or
     \item[$(b)$] $\vv\in D_i$, $\ww\in R_{i+1}$, for some $i\in\Z$, and $\sigma(\vv)=\ww$; or
     \item[$(c)$] $\vv\in R_{i+1}$, $\ww\in D_i$, for some $i\in\Z$, and $\sigma(\ww)=\vv$. 
\end{itemize}
It is not difficult to show that this is an equivalence relation on $\V_{\mathbb{G}}$. We then say that the quotient $\G:=\mathbb{G}/\sim$ is a $\Z$-periodic graph with periodicity cell $\K$ and pasting rule $\sigma$ (see Figure \ref{fig:defZ} for a concrete example).

With this definition of $\Z$-periodic graph it is also immediate to give a precise notion of $\Z$-periodic subsets $V$ of $\V_\G$. Indeed, if $\G$ is a $\Z$-periodic graph with periodicity cell $\K$, we say that $V\subseteq\V_\G$ is a $\Z$-periodic set in $\G$ if there exist $\vv_1,\dots,\vv_n\in\V_\K$ such that
\[
V=\bigcup_{i\in\Z}\left\{\vv_1^i,\dots,\vv_n^i\right\}\,,
\]
where $\vv_1^i,\dots,\vv_n^i\in\V_{\K_i}$ denote, for every $i\in\Z$, the copies of $\vv_1,\dots,\vv_n$ in $\K_i$.
\begin{remark}
	For a detailed discussion of this definition of $\Z$-periodic graphs and a comparison with the general definition of periodic graphs as in \cite[Definition 4.1.1]{BK}, we refer to \cite[Section 2 and Appendix A]{D-per}. 
\end{remark}

Clearly the two--dimensional square grid $\Q$ does not satisfy the above definition of $\Z$-periodic graphs. In this paper, we will often think of $\Q=(\V_{\Q}, \E_{\Q})$ as the subset of $\R^2$ with vertices on the lattice $\Z^2$ and edges between every couple of vertices at unitary distance in $\R^2$, so that
\begin{equation*}
\vv\in \V_{\Q}\cong (i,j)\in \Z^{2}\subset \R^{2},
\end{equation*}
and 
\begin{equation*}
\E_{\Q}=\E_{\Q}^{h}\cup \E_{\Q}^{v},
\end{equation*}
with
\begin{align*}
\E_{\Q}^{h}&:=\left\{(i,i+1)\times\{j\}\subset \R^{2},\,i,j\in\Z\right\}\\
\E_{\Q}^{v}&:=\left\{\{i\}\times(j,j+1)\subset\R^{2},\,i,j\in\Z \right\}.
\end{align*}
Sometimes it will also be convenient to think of $\Q$ as the union of infinitely many horizontal and vertical lines in $\R^2$, that is
\begin{equation*}
\Q=\Big(\bigcup_{j\in \Z} H_{j}\Big)\cup\Big(\bigcup_{i\in\Z} V_{i}\Big)\,,
\end{equation*}
with $H_{j}$ being the line of equation $y=j$ and $V_{i}$ that of equation $x=i$, for every $i,j\in\Z$.

\begin{figure}[t]
	\centering
	\includegraphics[width=0.9\textwidth]{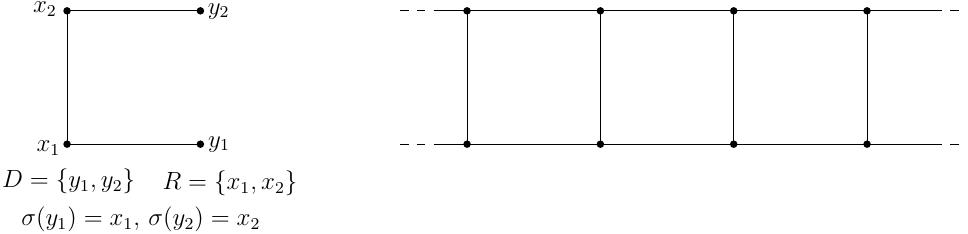}
	\caption{Example of a compact graph $\K$, the sets $D$, $R$, the function $\sigma$ (on the left) and the corresponding $\Z$-periodic graph (on the right) according to the definition described in Section \ref{sec:prel}.}
	\label{fig:defZ}
\end{figure}

\subsection{Gagliardo-Nirenberg inequalities}

Since they will be frequently used in the following, we recall  the following Gagliardo-Nirenberg inequalities, that  hold on every noncompact metric graph $\G$:

\begin{align}
\|u\|_{L^{p}(\G)}^{p} & \leq C_{p}\|u\|_{L^{2}(\G)}^{\f{p}{2}+1}\|u'\|_{L^{2}(\G)}^{\f{p}{2}-1},\qquad\forall\, u\in H^1(\G),\,\forall\, p>2, \label{GN-1d}\\
\|u\|_{L^{\infty}(\G)}^{2}&\leq C_\infty\|u\|_{L^2(\G)}\|u'\|_{L^2(\G)},\qquad\forall\, u\in H^1(\G)\,,\label{GN-inf}
\end{align}
with $C_p,C_\infty$ suitable 
constants depending on $\G$.
Furthermore, if $\G$ is the two-dimensional grid $\Q$, then also the next Gagliardo-Nirenberg inequality holds (see \cite[Theorem 2.3]{ADST}):
\begin{equation}
\label{GN-2d}
\|u\|_{L^{p}(\Q)}^{p}\leq M_{p}\|u\|_{L^{2}(\Q)}^{2} \|u'\|_{L^{2}(\Q)}^{p-2}\qquad \forall\, u\in H^{1}(\Q),\,\forall\, p>2\,,
\end{equation}
with $M_{p}>0$ depending only on $p$. 

A further Gagliardo-Nirenberg type inequality can be derived also for the sum of pointwise nonlinearities at the vertices of the graph.
\begin{lemma}
	\label{lem:mod-ineq-GN}
	For every $\G\in\mathbf{G}$ and $q\in(2,4)$ there exists $C>0$, depending on $\G$ and $q$, such that
	\begin{equation*}
	\sum_{\vv\in\V_{\G}}|u(\vv)|^{q}\leq C\left( \|u\|_{L^{q}(\G)}^{q}+\|u\|_{L^{2}(\G)}^{\f{q}{2}}\|u'\|_{L^{2}(\G)}^{\f{q}{2}}\right)\qquad \forall\, u\in H^{1}(\G)\,.
	\end{equation*}
\end{lemma}
\begin{proof}
	For every $\vv\in\V_{\G}$, let $\E_\vv$ be the set of all edges incident at $\vv$ and set $\displaystyle\ell_\vv:=\min\left\{1,\frac12\min_{e\in\E_\vv}|e|\right\}$. Note that, since $\G\in\mathbf{G}$ ensures that $\inf_{e\in\E_\G}|e|=:\alpha>0$, then $\alpha/2 \le \ell_\vv \le 1$ for every $\vv\in\V_\G$. For fixed $\vv$, let then $e_\vv\in\E_\vv$ be a given edge incident at $\vv$ and let $\vv$ be identified with $0$ along $e_\vv$. Then we have
	\[
	\begin{split}
	\left|\ell_\vv|u(\vv)|^q-\|u\|_{L^q(e_\vv\cap[0,\ell_\vv])}^q\right|=&\,\left|\int_{e_\vv\cap[0,\ell_\vv]}(|u(\vv)|^q-|u(y)|^q)\,dy\right|\\
	\leq&\,\ell_\vv\int_{e_\vv}|(|u(z)|^q)|'\,dz\leq q\ell_\vv\int_{e_\vv}|u(z)|^{q-1}|u'(z)|\,dz\,,
	\end{split}	
	\]
	which, recalling the lower bound on $\ell_\vv$,  yields
	\[
	|u(\vv)|^q\leq \frac2\alpha\|u\|_{L^q(e_\vv\cap[0,\ell_\vv])}^q+q\int_{e_\vv\cap[0,\ell_\vv]}|u(z)|^{q-1}|u'(z)|\,dz\,.
	\]
	Note that, by definition of $\ell_\vv$, if $\vv_1$ and $\vv_2$ are two different vertices in $\V_\G$, for every pair of edges $e_{\vv_1}\in\E_{\vv_1}$, $e_{\vv_2}\in\E_{\vv_2}$ the sets $e_{\vv_1}\cap[0,\ell_{\vv_1}]$, $e_{\vv_2}\cap[0,\ell_{\vv_2}]$ share at most one point (the one identified with $\ell_{\vv_1}$ and $\ell_{\vv_2}$ respectively). Hence, summing over $\vv\in\V_\G$ the above inequality and using H\"older inequality and \eqref{GN-1d} with $p=2(q-1)$ gives
	\[
	\begin{split}
	\sum_{\vv\in\V_\G}|u(\vv)|^q&\,\leq\frac2\alpha\|u\|_{L^q\left(\bigcup_{\vv\in\V_\G}(e_\vv\cap[0,\ell_\vv])\right)}^q+q\int_{\bigcup_{\vv\in\V_\G}(e_\vv\cap[0,\ell_\vv])}|u|^{q-1}|u'|\,dx\\
	&\, \leq \frac2\alpha\|u\|_{L^q(\G)}^q+q\int_\G|u|^{q-1}|u'|\,dx\leq\frac2\alpha\|u\|_{L^q(\G)}^q+q\|u\|_{L^{2(q-1)}(\G)}^{q-1}\|u'\|_{L^2(\G)}\\
	&\,\leq \frac2\alpha\|u\|_{L^q(\G)}^q+qC_{2(q-1)}^{1/2}\|u\|_{L^2(\G)}^{\frac q2}\|u'\|_{L^2(\G)}^{\frac q2}\,,
	\end{split}
	\]
	and we conclude.
\end{proof}
The last estimate we report in this section, given in the next lemma, will be particularly useful when proving Theorem \ref{thm:grid-Z-per} in Section \ref{sec:Vinf}.
\begin{lemma}
	\label{lem:ineq-G'}
	Let $\Q$ be the two--dimensional square grid and $\G'\subset \Q$ be a subgraph (i.e. $\V_{\G'}\subseteq\V_\Q$ and $\E_{\G'}\subseteq\E_\Q$) with $|\G'|>0$ and such that
	\begin{equation*}
	\min\left\{\sup_{j\in \Z}\#\left(\V_{\G'}\cap H_{j}\right), \sup_{j\in \Z}\#\left(\V_{\G'}\cap V_{j}\right) \right\}<+\infty\,,
	\end{equation*}
	where $H_j, V_j$ denote the $j$-th horizontal and vertical line in $\Q$ as above.
	Then, for every $q>2$, there exists $K>0$, depending on $\G'$ and $q$, such that  
	\begin{equation*}
	\|u\|_{L^{q}(\G')}^{q}\leq K\|u\|_{L^{2}(\Q)}\|u'\|_{L^{2}(\Q)}^{q-1}\qquad \forall\,u\in H^{1}(\Q).
	\end{equation*}
\end{lemma}
\begin{proof}
	Without loss of generality, suppose that
	\[
	\min\left\{\sup_{j\in \Z}\#\left(\V_{\G'}\cap H_{j}\right), \sup_{j\in \Z}\#\left(\V_{\G'}\cap V_{j}\right) \right\}=\sup_{j\in \Z}\#\left(\V_{\G'}\cap H_{j}\right)=:R\,.
	\]
	For every $j\in\Z$, every $e\in\E_{\G'}\cap H_j$ (if any) and every $x\in e$, we have
	\begin{equation}
	\label{u-q}
	|u(x)|^{q}\leq \int_{H_{j}}\left|\left(|u|^{q}\right)'\right|\,dy\leq q\int_{H_{j}}|u|^{q-1}|u'|\,dy\,,
	\end{equation}
	so that, integrating  on $e$ and summing over $e\in \E_{\G'}\cap H_j$ (recalling that there are at most $R$ of such edges), there results
	\begin{equation*}
	\label{int-G-H}
	\int_{\G'\cap H_{j}}|u|^{q}\,dy\leq qR\int_{H_{j}}|u|^{q-1}|u'|\,dy.
	\end{equation*}
	Summing over $j\in \Z$ and using the H\"older inequality we obtain
\[
\|u\|_{L^{q}(\G'\cap\bigcup_{j\in\Z}H_j)}^{q}\leq  q R\int_{\Q}|u|^{q-1}|u'|dy\leq q R \|u\|_{L^{2(q-1)}(\Q)}^{q-1}\|u'\|_{L^{2}(\Q)}
\]
	and applying \eqref{GN-2d} with $p=2(q-1)$ yields
	\begin{equation}
	\label{eq:uHj}
	\|u\|_{L^{q}(\G'\cap\bigcup_{j\in\Z}H_j)}^{q}\leq C\|u\|_{L^2(\Q)} \|u'\|_{L^2(\Q)}^{q-1}
	\end{equation}
	for a suitable $C>0$ depending on $q$ and $R$.
	
	Now, since $\E_{\G'}\subseteq\E_\Q$ and $\#\left(\V_{\G'}\cap H_j\right)\leq R$ for every $j$, the number of vertical edges in $\G'$ between $H_j$ and $H_{j+1}$ is bounded from above by $R$ uniformly on $j$. Arguing as above, for every such edge $e$ (if any) we have
	\[
	\|u\|_{L^q(e)}^q\leq\int_{e\cup H_j}|(|u|^q)'|\,dy\,,
	\] 
	so that, summing first over all vertical edges $e\in\E_{\G'}$ between $H_j$ and $H_{j+1}$, and then over $j\in\Z$, we obtain
	\[
	\|u\|_{L^q(\G'\cap\bigcup_{k\in\Z}V_k)}^q\leq R\int_{\Q\cap\bigcup_{j\in\Z}H_j}|(|u|^q)'|\,dy+\int_{\Q\cap\bigcup_{k\in\Z}V_k}|(|u|^q)'|\,dy\leq R\int_\Q|(|u|^q)'|\,dy\,,
	\]
	and, repeating the previous computations and combining with \eqref{eq:uHj}, we conclude.
\end{proof}

\section{General properties of $\D_{q,V}$}
\label{sec:D}
In this section we start the analysis of the minimization problem $\D_{q,V}$, developing a general framework that will then be used in the following sections to deal with the various families of graphs we are interested in. To this end, we introduce the following notation
\begin{equation}
\label{lambdaG}
\lambda(\G):=\inf_ {w\in H^1(\G)\setminus\{0\}}\frac{\|w'\|_{L^2(\G)}^2}{\|w\|_{L^2(\G)}^2}
\end{equation}
for the bottom of the spectrum of the operator $-d^2/dx^2$ on $\G$ coupled with homogeneous Kirchhoff conditions at every vertex of $\G$. 
\begin{remark}
	\label{rem:contD}
	Note that, for given $\G$, $V$ and $q$, the ground state level $\D_{q,V}(\,\cdot\,,\G):[0,+\infty)\to\R$ is continuous with respect to $\mu$. Indeed, since every $u\in H_\mu^1(\G)$ can be written as $u=\sqrt{\mu}v$ for some $v\in H_1^1(\G)$, we have
	\[
	\D_{q,V}(\mu,\G)=\inf_{v\in H_1^1(\G)}f_v(\mu)
	\]
	where
	\[
	f_v(\mu):=E_{q,V}(\sqrt{\mu}v,\G)=\f\mu2\|v'\|_{L^2(\G)}^2-\f{\mu^{\frac q2}}q\sum_{v\in V}|v(\vv)|^q\,.
	\]
	Since $q>2$, $f_v$ is a concave function of $\mu$ for every $v\in H_1^1(\G)$, thus showing that $\D_{q,V}$ is concave in $\mu$ (and therefore continuous).
\end{remark}
The first lemma of the section provides general upper and lower bounds on the ground state level $\D$.
\begin{lemma}
\label{lem:level}
Let $\G\in\mathbf{G}$ and $V\subseteq\V_{\G}$. For every $q\in(2,4)$ and every $\mu>0$ there results
\begin{equation}
\label{below-linear}
-\infty<\D_{q,V}(\mu,\G)\leq \f{\la(\G)}{2}\mu.
\end{equation}
\end{lemma}
\begin{proof}
	The upper bound in \eqref{below-linear} is a direct consequence of the definition of $\D$ and $\lambda(\G)$. As for the lower bound, note that it is enough to prove it when $V=\V_\G$, since
	\begin{equation*}
	\sum_{\vv\in V}|u(\vv)|^{q}\leq \sum_{\vv\in \V_{\G}}|u(\vv)|^{q}\,,\qquad \forall\,u\in H^{1}(\G),\, \forall\,V\subseteq\V_{\G}.
	\end{equation*} 
	Let then $V=\V_\G$. By Lemma \ref{lem:mod-ineq-GN} and inequality \eqref{GN-1d} with $p=q$, for every $u\in H^{1}_{\mu}(\G)$ we have
	\begin{equation*}
	\begin{split}
	E(u)&\geq\f12\|u'\|_{L^{2}(\G)}^{2}-\f{C}{q}\left(\|u\|_{L^{q}(\G)}^{q}+\|u\|_{L^{2}(\G)}^{\f{q}{2}}\|u'\|_{L^{2}(\G)}^{\f{q}{2}}\right)\\
	&\geq \f{1}{2}\|u'\|_{L^{2}(\G)}^{2}-\f{CC_{q}\mu^{\f{q}{4}+\f{1}{2}}}{q}\|u'\|_{L^{2}(\G)}^{\f{q}{2}-1}-\frac Cq\mu^{\f{q}{4}}\|u'\|_{L^{2}(\G)}^{\f{q}{2}}\,,
	\end{split}
	\end{equation*}
	where $C$ is the constant in Lemma \ref{lem:mod-ineq-GN}. The lower bound in \eqref{below-linear} then follows by $q<4$.
\end{proof}
In view of Lemma \ref{lem:level}, we give the next definition.   
\begin{definition}
\label{def:mu-star-q}
Let $\G\in\mathbf{G}$ and $V\subseteq \V_\G$. For every $q\in(2,4)$, set
\begin{equation*}
\mu^{*}_{q}:=\inf\left\{\mu\geq 0\,:\,\D_{q,V}(\mu,\G)<\frac{\lambda(\G)}{2}\mu\right\}.
\end{equation*}
\end{definition}
The quantity $\mu_q^*$ plays an important role in the characterization of the behaviour of $\D_{q,V}$ in terms of $\mu$, as described in the next lemma.
\begin{lemma}
\label{lem:mu-star}
Let $\G\in \mathbf{G}$ and $V\subseteq\V_\G$. Then, for every $q\in (2,4)$, it holds
\[
\mu^*_q\in [0,+\infty).
\]
Moreover, 
\begin{itemize}
\item[$(a)$] if $\mu\in(0,\mu^*_q]$, then $\D_{q,V}(\mu,\G)=\frac{\lambda(\G)}{2}\mu$, but $E_{q,V}(u,\G)>\frac{\lambda(\G)}2\mu$ for every $u\in H^1_\mu(\G)$ and every $\mu\in(0,\mu_q^*)$; 
\item[$(b)$] if $\mu>\mu^*_q$, then $\D_{q,V}(\mu,\G)<\frac{\lambda(\G)}{2}\mu$. 
\end{itemize}

\end{lemma}
\begin{proof}
We first prove that $\mu^*_q<+\infty$. To this end, take a vertex $\vv\in V$, that we identify in the following with $0$, and denote by $e_i$, $i=1,\dots,N$, the edges incident at $\vv$. Setting $\displaystyle \ell:=\min_{1\leq i\leq N}|e_i|$, we define a function $u_M\in H^1(\G)$ as 
\begin{equation*}
u_M(x)=
\begin{cases} M(\ell-x)\quad &\text{on}\quad e_i\cap[0,\ell],\quad \forall\,\, i=1,\dots,N\\
 0&\text{elsewhere}.
\end{cases}
\end{equation*}
Plainly, If $M\to+\infty$, we have $\|u_M\|_{L^2(\G)}^2\to +\infty$ and
\begin{equation*}
E(u_M)=\frac
12 N M^2\ell-\frac1q M^q\ell^q<0.
\end{equation*}
Hence, $\D(\mu)<0$ if $\mu$ is sufficiently large, entailing $\mu^*_q<+\infty$.

We now prove (b), i.e.  $\D(\mu)<\frac{\lambda(\G)}2\mu$ for $\mu>\mu^*_q$. Fix $\mu>\mu^*_q$ and observe that, by definition of $\mu^{*}_{q}$, there exists $\mu_{1}\in (\mu^{*}_{q},\mu)$ such that $\D(\mu_{1})<\frac{\lambda(\G)}{2}\mu_1$. In particular there exists $u_{1}\in H^{1}_{\mu_{1}}(\G)$ such that $E(u_{1})<\frac{\lambda(\G)}{2}\mu_1$. Since $\mu>\mu_1$,
\begin{equation}
\label{D-q-mu-rescaled}
\D(\mu)\leq E\left(\sqrt{\f{\mu}{\mu_{1}}}u_{1}\right)=\f{1}{2}\f{\mu}{\mu_{1}}\|u_{1}'\|_{L^{2}(\G)}^{2}-\f{1}{q}\left(\f{\mu}{\mu_{1}}\right)^{\f{q}{2}}\sum_{\vv\in V}|u_{1}(\vv)|^{q}\leq \f{\mu}{\mu_{1}}E(u_{1})<\frac{\lambda(\G)}{2}\mu,
\end{equation}
which proves (b).

Let us finally prove $(a)$. Clearly, if $\mu_q^*=0$ there is nothing to prove. Assume then $\mu^*_q>0$. Then by definition of $\mu^*_q$ it follows that $\D(\mu)\geq \frac{\lambda(\G)}2\mu$ for every $\mu \in (0,\mu^*_q)$, while $\D(\mu)\leq \frac{\la(\G)}2\mu$ by Lemma \ref{lem:level}. Hence, $\D(\mu)=\frac{\lambda(\G)}{2}\mu$ for every $\mu\in (0,\mu_q^*)$, and thus also at $\mu=\mu_q^*$ by Remark \ref{rem:contD}. Moreover, suppose by contradiction that there exists $\mu\in(0,\mu_q^*)$ and $u\in H^1_\mu(\G)$ such that $E(u)=\frac{\lambda(\G)}{2}\mu$.  We first note that $u\not\equiv0$ on $V$. Indeed, if this were not the case, then $u$ would satisfy $\|u'\|_{L^2(\G)}^2=\lambda(\G)\|u\|_{L^2(\G)}^2$. If $\lambda(\G) = 0$, this is impossible because $u \in H^1_\mu(\G)$ and $u \equiv 0$ on $V$ by assumption. If $\lambda(\G) >0$, $u$ would be an eigenfunction associated to the first eigenvalue $\lambda(\G)$ of the operator $-d^2/dx^2$ with homogeneous Kirchhoff conditions at every vertex of $\G$. Since such eigenfunctions (when they exist) do not vanish anywhere on $\G$, this would provide a contradiction. Fix now $\mu_1\in (\mu,\mu^*_q)$ and define $v:=\sqrt{\f{\mu_1}{\mu}}u\in H^1_{\mu_1}(\G)$. Arguing as in \eqref{D-q-mu-rescaled} and making use of $u\not\equiv0$ on $V$, one obtains that $\D(\mu_1)<\frac{\lambda(\G)}{2}\mu_1$, contradicting the fact that $\mu_1<\mu^*_q$ and concluding the proof. 
\end{proof}

As for the dependence of $\mu_q^*$ on the nonlinearity power $q$, we have the following general relation. 
\begin{lemma}
\label{lem:q<q1}
Let $\G\in\mathbf{G}$ and $V\subseteq\V_\G$. If $\mu^*_{\overline{q}}=0$ for some $\overline{q}>2$, then $\mu^*_q=0$ for every $q\in(2,\overline{q}]$.
\end{lemma}
\begin{proof}
Fix $\overline{q}>2$ and suppose that $\mu^*_{\overline{q}}=0$. By Lemma \ref{lem:mu-star}, $\D_{\overline{q}}(\mu)<\frac{\lambda(\G)}{2}\mu$ for every $\mu>0$, so that there exist $(u_{\mu})_{\mu>0}$ such that $u_{\mu}\in H^{1}_{\mu}(\G)$ and $E_{\overline{q}}(u_{\mu})<\frac{\lambda(\G)}{2}\mu$ for every $\mu>0$.  Note that, as in the final part of the proof of Lemma \ref{lem:mu-star}, $u_\mu\not\equiv0$ on $V$.
By Lemma \ref{lem:mod-ineq-GN} and \eqref{GN-1d}, this yields
\[
\frac12\|u_\mu'\|_{L^2(\G)}^2<\frac{\lambda(\G)}2\mu+CC_q\mu^{\frac{\overline{q}}{4}+\frac12}\|u_\mu'\|_{L^2(\G)}^{\frac{\overline{q}}{2}-1}+C\mu^{\frac{\overline{q}}{4}}\|u_\mu'\|_{L^2(\G)}^{\frac{\overline{q}}{2}}\,,
\]
which shows that $(u_\mu)_{\mu>0}$ is bounded in $H^1(\G)$ when $\mu$ is sufficiently small, and thus by \eqref{GN-inf} it follows that $\|u_{\mu}\|_{L^{\infty}(\G)}\to 0$ as $\mu\to 0$. Hence, $|u_{\mu}(\vv)|\leq 1$ uniformly on $\vv\in V$ as $\mu \to 0$. Therefore, if $q\in (2,\overline{q}]$, then 
\begin{equation*}
\f{1}{q}\sum_{\vv\in V}|u_{\mu}(\vv)|^{q}=\f{1}{\overline{q}}\sum_{\vv\in V}\left(\f{\overline{q}}{q}|u_{\mu}(\vv)|^{q-\overline{q}}\right)|u_{\mu}(\vv)|^{\overline{q}}\geq \f{1}{\overline{q}}\sum_{\vv\in V}|u_{\mu}(\vv)|^{\overline{q}},
\end{equation*}
which entails $\D_q(\mu)\leq E_{q}(u_{\mu})\leq E_{\overline{q}}(u_{\mu})<\frac{\lambda(\G)}{2}\mu$ for every $\mu>0$, i.e. $\mu^*_q=0$.
\end{proof}
The previous lemma suggests the following definition.
\begin{definition}
\label{q-star}
Let $\G\in\mathbf{G}$ and $V\subseteq\V_\G$. Set
\begin{equation*}
q^{*}:=\inf\left\{q\in(2,4)\,:\,\mu^{*}_{q}>0\right\}\,,
\end{equation*}
where, for every $q\in(2,4)$, $\mu^{*}_{q}$ is the number introduced in Definition \ref{def:mu-star-q}.
\end{definition}
Exploiting $q^*$ and $\mu_q^*$ it is then possible to provide a general description of the behaviour of $\D_{q,V}$ in terms of both $q$ and $\mu$.
\begin{lemma}
\label{lem:q-star}
Let $\G\in \mathbf{G}$ and $V\subseteq\V_\G$. 
\begin{itemize}
	\item[$(i)$] If $q^{*}=2$, then $\mu^{*}_{q}>0$ for every $q\in(2,4)$;
	\item[$(ii)$] if $q^{*}=4$, then $\mu^{*}_{q}=0$ for every $q\in(2,4)$;
	\item[$(iii)$] if $q^{*}\in(2,4)$, then $\mu^{*}_q=0$ for every  $q\in (2,q^{*})$ and $\mu^{*}_q>0$ for every $q\in (q^{*},4)$.
\end{itemize}
\end{lemma}
\begin{proof}
Points (i) and (ii) are obvious in view of Definition \ref{q-star}, as well as  $\mu^{*}_{q}=0$ if $q\in(2,q^{*})$ whenever $q^*\in(2,4)$. To conclude, let us show that $\mu^{*}_{q}>0$ if $q\in (q^{*},4)$ when $q^*<4$. Suppose by contradiction that there exists $q_{1}\in (q^{*},4)$ such that $\mu^{*}_{q_{1}}=0$. By Lemma \ref{lem:q<q1}, it would then follow that $\mu^{*}_{q}=0$ for every $q\in (2,q_{1}]$, that is $q^{*}\geq q_{1}$ by definition of $q^*$, providing the contradiction we seek.
\end{proof}

\begin{remark}
	It is easy to exhibit examples of $\G$ and $V$ for which $q^*\in(2,4)$ and $\mu_{q^*}^*>0$ (see e.g. Theorem \ref{thm:grid-Z-per}). However, in general we are not able to say whether this is always the case or $\mu_{q^*}^*=0$ for suitable choices of $\G$ and $V$ for which $q^*\in(2,4)$.
\end{remark}

Up to now, all the results of this section refer to properties of the ground state level $\D_{q,V}$. It is however evident that nothing we said so far is enough to ensure the existence of ground states with prescribed mass, as this requires 
compactness properties that depend on $\G$ and $V$. We conclude this section with a first general result in this direction, providing an existence criterion for ground states in the case $V$ is finite.
\begin{lemma}
	\label{lem:compact-crit}
	Let $\G\in\mathbf{G}$ and $V\subseteq\V_\G$ be such that $\#V<+\infty$. If, for some $q\in(2,4)$ and $\mu>0$,
	\begin{equation}
	\label{strict-below-linear}
	\D_{q,V}(\mu,\G)<\frac{\lambda(\G)}2\mu\,,
	\end{equation}
	then ground states of $E_{q,V}$ in $H_\mu^1(\G)$ exist.
\end{lemma}
\begin{proof}
	Let $\mu>0$ be such that \eqref{strict-below-linear} holds and consider a minimizing sequence $(u_{n})_{n}\subset H^{1}_{\mu}(\G)$ for $E$, i.e. $E(u_{n})\to\D(\mu)$ as $n\to+\infty$. By \eqref{strict-below-linear}, \eqref{GN-inf} and the fact that $d:=\#V$ is finite we have, for sufficiently large $n$,
	\begin{equation*}
	\f{\la(\G)}{2}\mu>E(u_{n})\geq \f{1}{2}\|u_{n}'\|_{L^{2}(\G)}^{2}-\f{2d\mu^{\f{q}{4}}}{q}\|u_{n}'\|_{L^{2}(\G)}^{\f{q}{2}},
	\end{equation*}
	i.e. $(u_{n})_{n}$ is bounded in $H^{1}(\G)$ since $q<4$. Thus, up to subsequences, $u_{n}\deb u$ weakly in $H^{1}(\G)$ and $u_{n}\to u$ strongly in $L^{\infty}_{\text{loc}}(\G)$. In particular, since $V$ has finite cardinality, $u_{n}(\vv)\to u(\vv)$ for every $\vv\in V$. Hence, by weak lower semicontinuity
	\begin{equation*}
	E(u)\leq \liminf_{n\to+\infty} E(u_{n})=\D(\mu)
	\end{equation*}
	and 
	\begin{equation*}
	m:=\|u\|_{L^{2}(\G)}^{2}\leq \liminf_{n\to+\infty} \|u_{n}\|_{L^{2}(\G)}^{2}=\mu\,.
	\end{equation*}
	Clearly, if we prove that $m=\mu$, then $u\in H_\mu^1(\G)$ and $\D(\mu)=E(u)$, that is $u$ is the required ground state. Let us thus show this arguing by contradiction.
	
	Suppose first that $u \equiv 0$ on $V$. Then $u_n(\vv)\to0$ as $n\to+\infty$ for every $\vv\in V$ and, since $V$ has finite cardinality, this implies
	\begin{equation*}
	\D(\mu)=\lim_{n\to+\infty} E(u_{n})=\f{1}{2}\lim_{n\to+\infty}\|u_{n}'\|_{L^{2}(\G)}^{2}\geq \f{\la(\G)}{2}\mu,
	\end{equation*}
	contradicting \eqref{strict-below-linear}. In particular, this rules out the case $m=0$. 
	
	Suppose then that $0<m<\mu$ and $u \not\equiv 0$ on $V$. Since $u_{n}\deb u$ weakly in $H^{1}(\G)$, as $n\to+\infty$ it follows that
	\begin{equation*}
	\begin{split}
	&\|u_{n}-u\|_{L^{2}(\G)}^{2}=\mu-m+o(1)\\
	&\|u_{n}'-u'\|_{L^{2}(\G)}^{2}=\|u_{n}'\|_{L^{2}(\G)}^{2}-\|u'\|_{L^{2}(\G)}^{2}+o(1)\\
	&|u_{n}(\vv)-u(\vv)|^{q}=|u_{n}(\vv)|^{q}-|u(\vv)|^{q}+o(1)\quad\text{uniformly on } \vv\in V\,,
	\end{split}
	\end{equation*}
	so that
	\begin{equation}
	\label{Dq-un-u}
	E(u_{n}-u)=E(u_{n})-E(u)+o(1)\,.
	\end{equation}
	On the one hand, since $\mu>m$ and $q>2$,
	\begin{equation*}
	\D(\mu)\leq E\left(\sqrt{\f{\mu}{m}}u\right)=\f{1}{2}\f{\mu}{m}\|u'\|_{L^{2}(\G)}^{2}-\f{1}{q}\left(\f{\mu}{m}\right)^{\f{q}{2}}\sum_{\vv\in V}|u(\vv)|^{q}<\f{\mu}{m}E(u),
	\end{equation*}
	entailing
	\begin{equation}
	\label{Dq-u->}
	E(u)>\f{m}{\mu}\D(\mu).
	\end{equation}
	On the other hand, since for large $n$, $\|u_n-u\|_{L^2(\G)}^2 < \mu$, and $q>2$,
	\begin{equation*}
	\D(\mu)\leq  E\left(\frac{\sqrt{\mu}}{{\|u_n-u\|_{L^2(\G)}}}\,(u_n-u)\right)<\f{\mu}{\|u_n-u\|_{L^2(\G)}^2}E(u_{n}-u),
	\end{equation*}
	from which it follows
	\begin{equation}
	\label{Dq-un-u-geq}
	\liminf_{n\to+\infty}E(u_{n}-u)\geq \f{\mu-m}{\mu}\D(\mu).
	\end{equation}
	Therefore, combining \eqref{Dq-un-u}, \eqref{Dq-u->} and \eqref{Dq-un-u-geq} we obtain
	\begin{equation*}
	\D(\mu)=\lim_{n\to+\infty}E(u_{n})\geq\liminf_{n\to+\infty}E(u_{n}-u)+E(u)>\f{\mu-m}{\mu}\D(\mu)+\f{m}{\mu}\D(\mu)=\D(\mu),
	\end{equation*}
	which is a contradiction. Hence, $m=\mu$ and the proof is complete.
\end{proof}

\section{Graphs with half--lines and $\Z$-periodic graphs: proof of Theorems \ref{thm:n-cpt-hl}--\ref{thm:Z-per}}
\label{sec:H+Zper}

This section is devoted to the proof of our main results concerning noncompact graphs with finitely many vertices and $\Z$-periodic graphs, namely Theorem \ref{thm:n-cpt-hl} and Theorem \ref{thm:Z-per} respectively.

Recall first that, if $\G$ is any of these graphs, then (see e.g. Lemma \ref{lem:bot-spec-0} below)
\begin{equation}
\label{eq:l=0}
\lambda(\G)=0\,.
\end{equation}
\begin{proof}[Proof of Theorem \ref{thm:n-cpt-hl}]
Since $\G$ contains finitely many vertices, by \eqref{eq:l=0} and Lemma \ref{lem:compact-crit}, to prove existence of ground states of $E_{q,V}$ in $H_\mu^1(\G)$ it is enough to show that $\D_{q,V}(\mu,\G)<0$. Moreover, to show that this is true for every $\mu>0$, by Lemma \ref{lem:mu-star} it is enough to show that $\mu_q^*=0$, that amounts to constructing a sequence of functions $(u_\mu)_{\mu>0}$ such that $u_\mu\in H_\mu^1(\G)$ and $E(u_\mu)<0$ as $\mu\to0^+$.

To this end, since $\G$ is a noncompact graph with finitely many edges, we can write $\G=\K\cup\bigcup_{i=1}^{N}\HH_{i}$, where $\K$ is the compact core of $\G$ and $\HH_i$, $i=1,\dots,N$, are the half--lines of $\G$. For every $\varepsilon>0$, let then $u_{\ep}\in H^{1}_{\mu}(\G)$ be defined as 
\begin{equation*}
u_{\ep}(x)=
    \begin{cases}
      \ep^{2}e^{-\ep^{q}x}\quad&\text{if }x\in\HH_{i},\text{ for some } i=1,\dots,N\\
      \ep^{2}\quad &\text{if }x\in\K\,.
    \end{cases}
\end{equation*}
Direct computations yield
\begin{equation*}
\mu=\left(\f{N}{2}+|\K|\ep^{q}\right)\ep^{4-q}
\end{equation*}
and
\begin{equation*}
E(u_{\ep})=\f{N}{4}\ep^{q+4}-\f{d}{q}\ep^{2q}\,,
\end{equation*}
with $d:=\#V$.
Since $\ep\to 0$ if and only if $\mu\to 0$ and $E(u_{\ep})<0$ if $\ep$ is sufficiently close to $0$ by $q<4$, we conclude.
\end{proof}

We now focus on $\Z$-periodic graphs. Before proving Theorem \ref{thm:Z-per}, we give the next existence criterion for ground states in the case of sets $V$ that are $\Z$-periodic in $\V_\G$ according to the definition given in Section \ref{sec:prel}.
\begin{lemma}
	\label{lem:exZper}
	Let $\G\in\mathbf{G}$ be a $\Z$-periodic graph and $V\subseteq\V_\G$ be a $\Z$-periodic set. If, for some $q\in(2,4)$ and $\mu>0$, there results $\displaystyle\D_{q,V}(\mu,\G)<0$,
	then ground states of $E_{q,V}$ in $H_\mu^1(\G)$ exist.
\end{lemma}
\begin{proof}
The argument of the proof is very similar to that of Lemma \ref{lem:compact-crit}. Let $\D(\mu)<0$ and $(u_n)_n\subset H_\mu^1(\G)$ be such that $E(u_n)\to \D(\mu)$ as $n\to+\infty$. Furthermore, exploiting the periodicity of both $\G$ and $V$, there is no loss of generality in assuming that each $u_n$ attains its $L^\infty$ norm on the same compact set $K\subset\G$ (independent of $n$). By Lemma \ref{lem:mod-ineq-GN}, \eqref{GN-1d} and $q<4$ we have that, up to subsequences, $u_n\rightharpoonup u$ weakly in $H^1(\G)$ as $n\to+\infty$. As in the proof of Lemma \ref{lem:compact-crit}, if $\|u\|_{L^2(\G)}^2=\mu$, we conclude. 

To show that $\|u\|_{L^2(\G)}\neq0$ note that, since $\|u_n\|_{L^\infty(\G)}=\|u_n\|_{L^\infty(K)}$ for every $n$ and the convergence of $u_n$ to $u$ is uniform on compact sets, if it were $u\equiv0$ we would have $u_n\to0$ in $L^\infty(\G)$. By Sobolev embeddings it would then follow, for a suitable constant $C>0$ independent of $n$,
\[
\sum_{\vv\in V}|u_n(\vv)|^q\leq\|u_n\|_{L^\infty(\G)}^{q-2}\sum_{\vv\in V}|u_n(\vv)|^2\leq C\|u_n\|_{L^\infty(\G)}^{q-2}\|u_n\|_{H^1(\G)}^2\to0\qquad\text{as }n\to+\infty\,,
\]
which by weak lower semicontinuity would yield
\[
\D(\mu)=\lim_{n\to+\infty}E(u_n)\geq\frac12\lim_{n\to+\infty}\|u_n'\|_{L^2(\G)}^2\geq0\,,
\]
contradicting $\D(\mu)<0$.

Finally, to exclude that $\|u\|_{L^2(\G)}^2\in(0,\mu)$ one argues exactly as in the final part of the proof of Lemma \ref{lem:compact-crit}, just recalling \eqref{eq:l=0} and noting that here the splitting
\[
E(u_n)=E(u_n-u)+E(u)+o(1)\qquad\text{as }n\to+\infty
\]
holds by the Brezis-Lieb Lemma \cite{BL}. 
\end{proof}

\begin{proof}[Proof of Theorem \ref{thm:Z-per}]
We divide the proof in three steps.

\smallskip
{\em Step 1: negativity of $\D(\mu)$.} Here we show that, for every $\Z$-periodic graph $\G$, every subset $V\subseteq\V_\G$, every $q\in(2,4)$ and every $\mu>0$, one always has $\D(\mu)<0$. Note that, by \eqref{eq:l=0} and Lemma \ref{lem:mu-star}, to prove this it is enough to construct functions $u_\eps$ such that $\|u_\eps\|_{L^2(\G)} \to 0$ and  $E(u_\eps)<0$ when $\eps \to 0$. Furthermore, it is clear that it is sufficient to obtain the result for $V=\{\overline{\vv}\}$ for a fixed $\overline{\vv}\in\V_\G$, since if $V$ contains more than one vertex the term $\sum_{\vv\in V}|u(\vv)|^q$ is not smaller than the value of $|u|^q$ at any given point of $V$.

Let then $V=\{\overline{\vv}\}$. Since $\G$ is a $\Z$-periodic graph according to the definition given in Section \ref{sec:prel}, let $\K, D, R$ be the corresponding periodicity cell and subsets of $\V_\K$. Let
\[
L_\K:=\left\{e\in \E_\K\,:\, \text{ there exists a unique }\vv\in D\,\text{ such that } e\succ \vv\right\}
\]
be the set of edges being incident at exactly one vertex $\vv$ in $D$, and set $l:=\min_{e\in L_\K}|e|$, $m:=\#L_\K$ and
\[
\widetilde{\K}=\K\setminus \bigcup_{e\in L_\K}\left(e\cap [0,l]\right),    
\]
where on each edge $e\in L_\K$ the corresponding vertex $\vv$ in $D$ is identified with $0$. 

For the sake of simplicity, let us first assume  that there are no edges in $\widetilde{\K}$ joining vertices in $D$. For every $\varepsilon>0$, let then $u_\ep\subset H^1_\mu(\G)$ be defined as
\[
u_\ep(x):=
\begin{cases}
\ep^2 e^{-\ep^q|(i+1)l-x|} \quad &\text{if}\,\, x\in e\cap[0,l],\,\,\text{for some}\,\,e\in L_{\K_i}\,\,\text{and}\,\,i\in\Z\\
     \ep^2 e^{-\ep^q|i|l} &\text{if}\,\, x\in \widetilde{\K}_i,\,\,\text{for some}\,\, i\in\Z\,,
\end{cases}
\]
where as usual $L_{\K_i}$ and $\widetilde{\K}_i$ are the copies of $L_\K$ and $\widetilde{\K}$ in $\K_i$ for every $i\in\Z$. Note that, exploiting the periodicity of $\G$ if necessary, there is no loss of generality in assuming that $u_\varepsilon(\overline{\vv})=\varepsilon^2$. As $\varepsilon\to0^+$ we then have
\[
\mu=\|u_\ep\|_{L^2(\G)}^2=m \ep^{4-q}+|\widetilde{\K}|\frac{e^{2l\ep^q}+1}{e^{2l\ep^q}-1}\ep^4= \left(m+\frac{|\widetilde{\K}|}{l}\right)\ep^{4-q}+o(\varepsilon^{4-q})
\]
and 
\begin{equation*}
    E(u_\ep)=\frac{1}2\|u_\varepsilon'\|_{L^2(\G)}^2-\frac1q|u_\varepsilon(\overline{\vv})|^q= \frac{m}{2}\ep^{q+4}-\frac{\ep^{2q}}{q}<0
\end{equation*}
since $q<4$. This proves the claim of Step 1 when there is no edge in $\widetilde{\K}$ joining vertices of $D$. To adapt the construction to cover this second case, however, it is enough to set $u_\varepsilon\equiv \ep^2 e^{-\ep^q|(i+1)l|}$ on each of such edges in $\widetilde{\K}_i$ and repeat the previous computations. 

\smallskip
{\em Step 2: proof of (i) and (ii).} If $V$ is such that $\#V<+\infty$, the existence of ground states of $E$ in $H_\mu^1(\G)$ for every $q\in(2,4)$ and every $\mu>0$ follows by $\D(\mu)<0$, \eqref{eq:l=0} and Lemma \ref{lem:compact-crit}. The same is true if $V$ is a $\Z$-periodic subset of $\V_\G$, simply using Lemma \ref{lem:exZper} in place of Lemma \ref{lem:compact-crit}.

\smallskip
{\em Step 3: proof of (iii).} Given $\G$, we  construct a set $V\subset \V_\G$ with $\#V=+\infty$ such that ground states of $E_{q,V}$ in $H_\mu^1(\G)$ never exist, independently of the values of $q\in(2,4)$ and $\mu$.

To this end, let $\overline{\vv}$ be a fixed vertex of the periodicity cell $\K$ of $\G$,  let $\overline{\vv}_i$ be the corresponding vertex in $\K_i$, for every $i\in\Z$, and call $\overline V$ the union of all the $\overline \vv_i$'s.
Let $a_n:= n(n+1)$, $n\in \N$, and set
\[
V:=\{\overline{\vv}_i \mid i \ne  a_n,\; \forall n \in\N \}.
\]
Note that, for every $n \ge 0$, between $\overline\vv_{a_n}$ and $\overline\vv_{a_{n+1}}$ (that are not in $V$) there are $2n+1$ consecutive copies of $\vv$, all in $V$.

We now show that, if $u\in H_\mu^1(\G)$ is such that $u>0$ everywhere on $\G$, then there exists $w\in H_\mu^1(\G)$ such that $E_{q,V}(w,\G)<E_{q,V}(u,\G)$. This, together with the fact that ground states, when they exist, do not vanish on $\G$, implies that for this choice of $V$ ground states never exist. 

Let then $u\in H_\mu^1(\G)$, $u>0$ on $\G$, be fixed. Since by construction $V$ is strictly contained in the set $\overline V$ and $u>0$, we have
\[
\sum_{\vv\in V}|u(\vv)|^q < \sum_{\vv\in\overline{V}}|u(\vv)|^q,
\]
both sums being finite by Lemma \ref{lem:mod-ineq-GN}. Therefore there exists $N\in \N$ such that 
\begin{equation}
\label{eq:ass1}
\sum_{i=-N}^N|u(\overline{\vv}_i)|^q>  \sum_{\vv\in V}|u(\vv)|^q.
\end{equation}
Let now $w\in H_\mu^1(\G)$ be the function obtained by composing $u$ with the discrete translation on $\G$ that, for every $i$, maps each point in $\K_i$ into its copy in $\K_{i+c_N}$ (where $c_N$ is the midpoint between $a_N$ and $a_{N+1}$).

By definition, $\|w'\|_{L^2(\G)}=\|u'\|_{L^2(\G)}$ and, since $\overline \vv_i \in V$ for every $i = a_N+1, \dots, a_{N+1}-1$, by \eqref{eq:ass1},
\[
\sum_{\vv\in V}|w(\vv)|^q >\sum_{i=a_N+1}^{a_{N+1}-1}|w(\overline{\vv}_i)|^q=\sum_{i=-N}^N|u(\overline{\vv}_i)|^q >\sum_{\vv\in V}|u(\vv)|^q,
\]
in turn implying that $E(w)<E(u)$ and concluding the proof.

\end{proof}

\section{The grid $\Q$ with $V$  finite: proof of Theorem \ref{thm:Vfin1} and Proposition \ref{prop:rad}}
\label{sec:QVfin}

Here we begin our analysis of the minimization problem $\D_{q,V}$ on the two--dimensional square grid $\Q$, focusing on sets $V\subset\V_\Q$ with $\#V<+\infty$. Since this section is rather long, it is divided in two subsections: in the first one we introduce and discuss a new minimization problem on the half-line $\R^+$, whereas in the second one we show how that is related to our original problem on $\Q$ and we use it to prove Theorem \ref{thm:Vfin1} and Proposition \ref{prop:rad}.

\subsection{A new minimization problem on $\R^+$}
Let $g:\Rp\to [4,+\infty)$ be defined as
\begin{equation}
\label{eq:gtilde}
g(x):=\begin{cases}
4\quad & \text{if }x \in [0,1]\\
4(2x-1)& \text{if }x>1
\end{cases}
\end{equation}
and set
\[
\begin{split}
H^1(\Rp, g\,dx):=\Big\{v:\Rp\to\R\,:\,&\,\|v\|_{L^2(\Rp,g\,dx)}^2+\|v'\|_{L^2(\Rp,g\,dx)}^2\\
&\qquad:=\int_{\Rp}|v(x)|^2g(x)\,dx+\int_{\Rp}|v'(x)|^2g(x)\,dx<+\infty\Big\}\,.
\end{split}
\]
Clearly, by \eqref{eq:gtilde}, $H^1(\Rp,g\,dx)\subset H^1(\Rp)$ and  is a Hilbert space endowed with the norm above.

For fixed $\alpha>0$, define $\widetilde{E}_{q,\alpha}:H^1(\Rp, g\,dx)\to\R$ as
\begin{equation}
\label{red-en}
\widetilde{E}_{q,\alpha}(v):=\f{1}{2}\|v'\|_{L^2(\Rp,g\,dx)}^2-\f{\alpha}{q} |v(0)|^{q}
\end{equation}
and set
\begin{equation*}
    X_\mu:=\left\{v\in H^1(\Rp, g\,dx) \mid 0<\|v\|_{L^2(\Rp,g\,dx)}^2\leq \mu\right\}
\end{equation*}
and 
\begin{equation*}
    \widetilde{\D}_{q,\alpha}(\mu):=\inf_{v\in X_\mu}\widetilde{E}_{q,\alpha}(v)\,.
\end{equation*}
Note that
\begin{equation}
	\label{eq:Dtildeleq0}
	\widetilde{\D}_{q,\alpha}(\mu)\leq\inf_{u\in X_\mu}\frac12\|u'\|_{L^2(\Rp,g\,dx)}^2=0\,.
\end{equation}
As usual, we say that $u\in X_\mu$ is a ground state of $\widetilde{E}_{q,\alpha}$ on $X_\mu$ if  
\begin{equation*}
    \widetilde{E}_{q,\alpha}(u)=\widetilde{\D}_{q,\alpha}(\mu).
\end{equation*}
The aim of this subsection is to prove the following result.
\begin{proposition}
	\label{prop:Dq-to-0}
	For every $q\in(2,4)$ and every $\alpha>0$ there exists a threshold $\widetilde{\mu}>0$, depending on $q$ and $\alpha$, such that for every $\mu\in(0,\widetilde{\mu})$ there results 
	\[
	\widetilde{E}_{q,\alpha}(u)>0\qquad \forall\,u\in X_\mu\,.
	\]
\end{proposition}
Before proving Proposition \ref{prop:Dq-to-0}, we need the next two lemmas.

\begin{lemma}
	\label{lem:dq<o->ex}
	Let $q\in(2,4)$, $\alpha>0$ and $\mu>0$ be fixed. If  $\widetilde{\D}_{q,\alpha}(\mu)<0$, then there exists a ground state of $\widetilde{E}_{q,\alpha}$ on $X_\mu$. Furthermore, if $u\in X_\mu$ is a ground state of $\widetilde{E}_{q,\alpha}$, then $\|u\|_{L^2(\Rp, g\,dx)}^2=\mu$ and there exists $\lambda\in\R$ such that
	\begin{equation}
	\label{EL-eq-Dtilde}
	\begin{cases}
	\left(u'\, g\right)'=\la u\, g &  \text{on }(0,1) \cup (1,+\infty)\\
	4u_+'(0)=-\alpha|u(0)|^{q-2}u(0)\,.
	\end{cases}
	\end{equation}
\end{lemma}
\begin{proof}
	Note first that, by the Gagliardo--Nirenberg inequality \eqref{GN-inf} on $\R^+$ and  $g(x)\geq 1$ for every $x\in \Rp$, there results for every $u\in X_\mu$
	\begin{equation*}
	\begin{split}
	\widetilde{E}_{q,\alpha}(u)&\geq \f{1}{2}\|u'\|_{L^2(\Rp,g\,dx)}^2-\f{C_\infty^{q/2}\alpha}{q}\|u\|_{L^2(\Rp)}^{q/2}\|u'\|_{L^2(\Rp)}^{q/2}\\
	&\geq \f{1}{2}\|u'\|_{L^2(\Rp, g\,dx)}^2-\f{C_\infty^{q/2}\alpha}{q}\|u\|_{L^2(\Rp,g\,dx)}^{q/2}\|u'\|_{L^2(\Rp,g\,dx)}^{q/2}\\
	&\geq\f{1}{2}\|u'\|_{L^2(\Rp,g\,dx)}^2-\f{C_\infty^{q/2}\alpha\mu^{q/4}}{q}\|u'\|_{L^2(\Rp,g\,dx)}^{q/2}\,,
	\end{split}
	\end{equation*}
	which shows that $\widetilde{\D}_{q,\alpha}(\mu)>-\infty$ since $q<4$.
	
	Let then $(u_{n})_{n}\subset X_\mu$ be a minimizing sequence for \eqref{red-en}, namely 
	\begin{equation*}
	0<\|u_n\|_{L^2(\Rp,g\,dx)}^2\leq \mu\quad\text{and}\quad \widetilde{E}_{q,\alpha}(u_{n})\to\widetilde{\D}_{q,\alpha}(\mu)\qquad\text{as }n\to+\infty.
	\end{equation*}
	The previous computation shows that $(u_n)_n$ is bounded in $H^1(\Rp,g\,dx)$, and thus also in $H^{1}(\Rp)$. Hence, up to subsequences, there exists $u\in H^{1}(\Rp)$ such that $u_{n}\deb u$ weakly in $H^{1}(\Rp)$ and $u_{n}(0)\to u(0)$ as $n\to+\infty$. Moreover, since
	\begin{equation*}
	\int_{\Rp}|u_{n}(x)|^{2}\,g(x)\,dx\leq \mu,
	\end{equation*}
	there exists $v\in L^{2}(\Rp)$ such that, up to subsequences, $\sqrt{g}u_{n}\deb v$ weakly in $L^{2}(\Rp)$ as $n\to+\infty$, so that by the uniqueness of the distributional limit it follows that, up to subsequences,  $\sqrt{g}u_{n}\deb \sqrt{g}u$ weakly in $L^{2}(\Rp)$ as $n\to+\infty$ and 
	\begin{equation*}
	\|u\|_{L^2(\Rp,g\,dx)}^2\leq\liminf_{n\to+\infty}\|u_n\|_{L^2(\Rp,g\,dx)}^2\leq\mu.
	\end{equation*}
	Arguing analogously we have also that, up to subsequences, $\sqrt{g}u_{n}'\deb \sqrt{g}u'$ weakly in $L^{2}(\Rp)$ as $n\to+\infty$ and
	\begin{equation*}
	\|u'\|_{L^2(\Rp,g\,dx)}^2\leq\liminf_{n\to+\infty}\|u_n'\|_{L^2(\Rp,g\,dx)}^2.
	\end{equation*}
	Therefore, if we show that $u\not\equiv0$ on $\R^+$, we obtain $u\in X_\mu$ and
	\begin{equation*}
	\widetilde{\D}_{q,\alpha}(\mu)\leq \widetilde{E}_{q,\alpha}(u)\leq \liminf_{n\to+\infty}\widetilde{E}_{q,\alpha}(u_{n})=\widetilde{\D}_{q,\alpha}(\mu)\,,
	\end{equation*}
	i.e. $u$ is the desired ground state. 
	
	Assume thus by contradiction that $u\equiv 0$. This implies that $u_{n}(0)\to 0$ as $n\to+\infty$, in turn yielding
	\begin{equation*}
	\widetilde{\D}_{q,\alpha}(\mu)=\lim_{n\to+\infty}\widetilde{E}_{q,\alpha}(u_{n})\geq \liminf_{n\to+\infty}\f{1}{2}\int_{\Rp}|u_{n}'(x)|^{2}\,g(x)\,dx\geq 0,
	\end{equation*}
	which contradicts the fact that $\widetilde{\D}_{q,\alpha}(\mu)<0$ by assumption. Hence, if $\widetilde{\D}_{q,\alpha}(\mu)<0$, then a ground state exists.
	
	Let now $u\in X_\mu$ be such that $\widetilde{E}_{q,\alpha}(u)=\widetilde{\D}_{q,\alpha}(\mu)$. Note that $u(0)\neq0$, since if this were not the case we would have $\widetilde{\D}_{q,\alpha}(\mu)=E_{q,\alpha}(u)>0$, contradicting \eqref{eq:Dtildeleq0}. Assume that $\|u\|_{L^2(\Rp,g\,dx)}^2:=m<\mu$. Then, setting $\displaystyle v:=\sqrt{\frac \mu m}u$, we obtain $\|v\|_{L^2(\Rp,g\,dx)}^2=\mu$, so that in particular $v\in X_\mu$, and
	\[
	\widetilde{E}_{q,\alpha}(v)=\frac \mu m\frac{1}{2}\|u'\|_{L^2(\Rp,g\,dx)}^2-\left(\frac{\mu}{m}\right)^{\frac q2}\frac{\alpha}{q}|u(0)|^q<\frac\mu m \widetilde{E}_{q,\alpha}(u)\leq\widetilde{\D}_{q,\alpha}(\mu)\,,
	\]
	which is a contradiction. Hence, if $u\in X_\mu$ is a ground state of $\widetilde{E}_{q,\alpha}$, then $\|u\|_{L^2(\Rp,g\,dx)}^2=\mu$. This also shows that 
	\[
	\widetilde{\D}_{q,\alpha}(\mu)=\inf\left\{\widetilde{E}_{q,\alpha}(u)\,:\, u\in H^1(\Rp,g\,dx),\,\|u\|_{L^2(\Rp,g\,dx)}^2=\mu\right\}\,,
	\]
	and writing the Euler-Lagrange equations of this minimization problem proves the existence of $\lambda\in\R$ for which \eqref{EL-eq-Dtilde} holds.
\end{proof}

\begin{lemma}
\label{lem:Dq<0-kin-bound}
Let $q\in(2,4)$ and $\alpha>0$ be fixed. If, for every $\mu$, $u_\mu\in X_\mu$ is such that $\widetilde{E}_{q,\alpha}(u_\mu)<0$, then as $\mu\to0^+$ there results
\begin{equation}
\label{eq:umu0}
|u_{\mu}(0)|=O\left(\mu^{\frac 1{4-q}}\right)
\end{equation}
and
\begin{equation}
\label{eq:u01}
\int_{0}^{1}|u_\mu(x)|^{2}\,dx=|u_\mu(0)|^{2}+o\left(|u_\mu(0)|^{2}\right)\,.
\end{equation}
\end{lemma}
\begin{proof}
Combining \eqref{GN-inf} on $\Rp$ with $g\geq 1$ and $\widetilde{E}_{q,\alpha}(u_\mu)<0$, we have
\begin{equation*}
|u_\mu(0)|^4\leq C_\infty^2\|u_\mu\|_{L^2(\Rp)}^2\|u_\mu'\|_{L^2(\Rp)}^2\leq C_\infty^2\mu\|u_\mu'\|_{L^2(\Rp,g\,dx)}^2<\f{2\alpha C_\infty^2}{q}\mu|u_\mu(0)|^{q}
\end{equation*}
that gives \eqref{eq:umu0}.
Moreover, since for every $x \in [0,1]$,
\[
|u_\mu(x)|^2  = |u_\mu(0)|^2 + |u_\mu(x) -u_\mu(0)|^2 + 2 u_\mu(0)(u_\mu(x)-u_\mu(0)) 
\]
and
\[
|u_\mu(x)-u_\mu(0)| \le \int_0^x|u_\mu'(t)|\,dt \le  \left(\int_0^1|u_\mu'|^2 g\,dt\right)^{1/2} <\left(\frac{2\alpha}q |u_\mu(0)|^q\right)^{1/2} =o( |u_\mu(0)|),
\]
we obtain, as $\mu \to 0$,
\[
|u_\mu(x)|^2 = |u_\mu(0)|^2 + o( |u_\mu(0)|^2).
\]
Integrating over $[0,1]$ yields \eqref{eq:u01}.
\end{proof}

We are now in a position to prove Proposition \ref{prop:Dq-to-0}.

\begin{proof}[Proof of Proposition \ref{prop:Dq-to-0}]
We prove the claim by contradiction. Suppose that for every $\mu>0$ there exists $w = w_\mu\in X_\mu$ such that $\widetilde{E}_{q,\alpha}(w)\leq0$. By \eqref{eq:Dtildeleq0} and Lemma \ref{lem:dq<o->ex}, this implies that there exists $u = u_\mu\in X_\mu$ such that $\widetilde{\D}_{q,\alpha}(\mu)=\widetilde{E}_{q,\alpha}(u)\leq 0$, $\|u\|_{L^2(\Rp,g\,dx)}^2=\mu$ and $u$ solves \eqref{EL-eq-Dtilde} for a suitable $\lambda\in\R$. As usual, with no loss of generality, we can take $u \ge 0$ on $\Rp$.

As $\widetilde{E}_{q,\alpha}'(u)u= 0$,
\begin{equation}
\label{stima1}
\int_{\Rp}|u'|^{2}g\dx =  \alpha u(0)^{q }- \lambda\mu
\end{equation}
from which, using $\widetilde{E}_{q,\alpha}(u)\leq 0$, we also see that $\lambda >0$. Then  \eqref{EL-eq-Dtilde} implies that  $gu'$ is increasing on $\R^+$. Therefore it has a limit as $x \to +\infty$ and it is easily seen that this limit must be zero. This also implies that $u'(x) \le 0$ in $[0,+\infty)$, so that, by \eqref{EL-eq-Dtilde}, we have
\[
(gu)'' = 2g'u'+gu'' = 2g'u' +\lambda gu - g'u' = g'u' + \lambda gu \le \lambda gu\qquad\text{on } [1,+\infty)
\]
which shows (by a standard application of the Maximum Principle) that $gu$ tends to $0$ (exponentially) as $x\to +\infty$. Having established that $gu$ and $gu'$ tend to zero as $x \to +\infty$, the following computations are fully justified.

Multiplying the equation in  \eqref{EL-eq-Dtilde} by $gu'$ and integrating yields, keeping into account the values of $g(0)$ and $u'(0)$,
\[
\int_{\R^+} (gu')'gu'\dx = \frac12\int_{\R^+} \frac{d}{dx}|gu'|^2\dx = -\frac12 |g(0)u'(0)|^2 = -8|u'(0)|^2 = -\frac{\alpha^2}{2}u(0)^{2q-2}
\]
for the left-hand-side. As for the right-hand-side,
\begin{align*}
\lambda\int_{\R^+} uu'g^2\dx &= \frac{\lambda}2\int_{\R^+}  g^2\frac{d}{dx}|u|^2\dx = -\frac{\lambda}2 |g(0)u(0)|^2 - \lambda\int_{\R^+} u^2gg'\dx \\
&= -8\lambda u(0)^2 -8 \lambda \int_1^{+\infty} u^2 g\dx = -8\lambda u(0)^2 - 8\lambda\mu + 32\lambda\int_0^1 u^2\dx
\end{align*}
since $g' \equiv 0$ on $(0,1)$ and $g'\equiv 8$ on $(1,+\infty)$.

Equating both sides and using Lemma \ref{lem:Dq<0-kin-bound} we then obtain, as $\mu \to 0$,
\[
\frac{\alpha^2}{16} u(0)^{2q-2} = \lambda u(0)^2   + \lambda\mu - 4\lambda \int_0^1  u^2\dx = \lambda\mu -3\lambda u(0)^2 + o(\lambda u(0)^2) = \lambda\mu + o(\lambda\mu)
\]
or
\[
\lambda\mu = \frac{\alpha^2}{16} u(0)^{2q-2} + o(u(0)^{2q-2}) = o(u(0)^q) \qquad\text{as } u(0) \to 0,
\]
since $q>2$. Now from this and \eqref{stima1} we see that
\[
\widetilde{E}_{q,\alpha}(u) = \frac12\int_{\R^+} |u'|^2g\dx - \frac{\alpha}{q} u(0)^q = \frac\alpha{2} u(0)^{q }- \frac12 \lambda\mu -  \frac{\alpha}{q} u(0)^q = \alpha\left(\frac12 - \frac1q\right) u(0)^q + o(u(0)^q) >0
\]
as $u(0) \to 0$, contradicting the assumption $\widetilde{E}_{q,\alpha}(u) \le 0$ and concluding the proof.
\end{proof}

\subsection{The minimization problem on $\Q$}
\label{subsec:Q} 
In this second part of the section we exploit the analysis of the previous part to prove our main results concerning the minimization problem $\D_{q,V}$ on $\Q$ when $V$ is finite. To this end, we first need to recall the notion of radial functions on the two--dimensional grid $\Q$.
\begin{definition}
	\label{u-rad}
	A function $f:\Q\to \R$ is radial with respect to the vertex $\vv\in\V_\Q$ if there exists $\widetilde{f}:\Rp\to\R$ such that $f(x)=\widetilde{f}(d(x,\vv))$ for every $x\in \Q$, where $d(x,\vv)$ denotes the (shortest path) distance between the point $x$ and the vertex $\vv$. Furthermore, we say that $f$ is radially decreasing on $\Q$ with respect to $\vv$ if it is radial with respect to $\vv$ and $\widetilde{f}$ is decreasing on $\Rp$.
\end{definition}
\begin{remark}
	\label{rem:rad}
	Note that, if $f\in L^p(\Q)$ is radial with respect to $\vv$, one can write $\|f\|_{L^p(\Q)}^p$ in terms of the associated $\widetilde{f}$ in the following way. For every $n\in\N$, denote by $B_n$ the open ball in $\Q$ of radius $n$ centered at $\vv$. It is not difficult to prove (e.g. by induction) that the number of edges of $\Q$ belonging to $B_{n+1}\setminus B_n$ is given by $4(2n+1)$. Then $f$ being radial implies
	\[
	\|f\|_{L^p(\Q)}^p=\sum_{n=0}^{+\infty} \sum_{e\in \E_\Q\cap B_{n+1}\setminus B_n}\int_e|f|^p\,dx=\sum_{n=0}^{+\infty}4(2n+1)\int_n^{n+1}|\widetilde{f}|^p\,dx\,.
	\]
	
\end{remark}
The importance of radial functions in our setting is given by the next rearrangement-type result.
\begin{lemma}
	\label{lem:rad1}
	Let $u\in H^1(\Q)$, $u\geq0$ on $\Q$, and $\vv\in \V_{\Q}$. Then there exists $w\in H^1(\Q)$, $w\ge0$ on $\Q$, such that $w$ is radial with respect to $\vv$ and
	\begin{equation}
	\label{rad-f-grid}
	\|w\|_{L^2(\Q)}^2\leq \|u\|_{L^2(\Q)}^2,\quad \|w'\|_{L^2(\Q)}^2\leq \|u'\|_{L^2(\Q)}^2\quad \text{and}\quad w(\vv)=u(\vv).
	\end{equation}
	Moreover, if $u$ is not radial with respect to $\vv$, then the inequalities in \eqref{rad-f-grid} are strict.
\end{lemma}
\begin{proof}
	Let $\vv\in\V_\Q$ be fixed, take $u\in H^1(\Q)$, $u\ge0$ on $\Q$, and define $w:\Q\to\R$ to be, at every point $x\in\Q$, the mean value of $u$ on the sphere in $\Q$ centered at $\vv$ of radius $d(x,\vv)$, that is (recalling Remark \ref{rem:rad})
	\begin{equation}
	w_e(x):=\f{1}{4(2n+1)}\sum_{f\in \E_\Q\cap B_{n+1}\setminus B_n} u_f(x)\qquad\text{if }x\in e\text{ and }e\in\E_\Q\cap B_{n+1}\setminus B_n,\text{ for some }n\in\N\,,
	\end{equation}
	where every edge  of $\Q$ is identified with the interval $[0,1]$ so that $x=0$ corresponds to its closest vertex to $\vv$. By definition, $w\geq0$ on $\Q$, it is continuous and radial with respect to $\vv$ and $w(\vv)=u(\vv)$. Moreover, 
	\begin{equation*}
	\begin{split}
	\|w\|_{L^2(\Q)}^2&=
	\sum_{n=0}^{+\infty}\sum_{e\in \E_\Q\cap B_{n+1}\setminus B_n}\int_e w_e^2\,dx=\sum_{n=0}^{+\infty}4(2n+1)\int_0^1\bigg(\f{1}{4(2n+1)}\sum_{f\in\E_\Q\cap B_{n+1}\setminus B_n} u_f(x)\bigg)^2\,dx\\
	&\leq \sum_{n=0}^{+\infty}4(2n+1)\int_0^1 \bigg(\f{1}{4(2n+1)}\sum_{f\in \E_\Q\cap B_{n+1}\setminus B_n} u_f^2(x)\bigg)\,dx\\
	&=\sum_{n=0}^{+\infty}\sum_{f\in\E_\Q\cap B_{n+1}\setminus B_n}\int_f u_f^2\,dx=\|u\|_{L^2(\Q)}^2,
	\end{split}
	\end{equation*}
	where we made use of the  inequality
	\begin{equation}
	\label{discr-Jens}
	\left(\sum_{i=1}^{n}x_{i}\right)^{2}\leq n\sum_{i=1}^{n}x_{i}^{2},\qquad\forall (x_i)_{i=1}^n\subset\R,\,n\in\N\,.
	\end{equation}
	Arguing analogously on $w'$ one also obtains $\|w'\|_{L^2(\Q)}\leq\|u'\|_{L^2(\Q)}$. Moreover, since the equality in \eqref{discr-Jens} is realized if and only if $x_{i}=x_{j}$ for every $i,j=1,\dots,n$, equalities in \eqref{rad-f-grid} are realized if and only if $u$ (and thus $u'$) is radial with respect to the vertex $\vv$.
\end{proof}
We can now prove Theorem \ref{thm:Vfin1} and Proposition \ref{prop:rad}. To this end, recall first that (see Lemma \ref{lem:bot-spec-0} below) 
\begin{equation}
	\label{eq:lQ=0}
	\la(\Q)=0\,.
\end{equation}
\begin{proof}[Proof of Theorem \ref{thm:Vfin1}]
To prove the results about $\EE_{q,V}$ and the ground states, by \eqref{eq:lQ=0} and Lemmas \ref{lem:mu-star}--\ref{lem:compact-crit}, it it enough to show that $\mu_q^*>0$ for every $V\subset \V_\Q$ with $\#V<+\infty$ and $q\in(2,4)$. To do this, given $V$ and $q$, we prove that for every $u\in H_\mu^1(\Q)$, $u\geq0$ on $\Q$, there exists $v:\Rp\to\R$ such that $v\in X_\mu$ and
\begin{equation}
\label{eq:v<u}
\widetilde{E}_{q,d}(v)\leq E_{q,V}(u,\Q)\,,
\end{equation}
where $d:=\#V$ and $X_\mu$, $\widetilde{E}_{q,d}$ are as in the previous subsection. Note that, once we have this, Proposition \ref{prop:Dq-to-0} implies that $E_{q,V}(u,\Q)>0$ for every $u\in H_\mu^1(\Q)$ as soon as $\mu$ is small enough, that is exactly $\mu_q^*>0$, and we are done.

To construct $v$ satisfying \eqref{eq:v<u}, note first that, since $\#V=d$,
\[
E_{q,V}(u,\Q)\geq\frac12\|u'\|_{L^2(\Q)}^2-\frac dq|u(\overline{\vv})|^q\,,
\]
where $\overline{\vv}:=\arg\max_{\vv\in V}|u(\vv)|^q$. Hence, by Lemma \ref{lem:rad1} there exists $w\in H^1(\Q)$, radial with respect to $\overline{\vv}$ such that $\|w\|_{L^2(\Q)}^2\leq\|u\|_{L^2(\Q)}^2=\mu$ and 
\begin{equation}
\label{eq:wu}
\frac12\|w'\|_{L^2(\Q)}^2-\frac dq|w(\overline{\vv})|^q\leq E_{q,V}(u,\Q)\,.
\end{equation}
Moreover, since the function $g$ as defined in \eqref{eq:gtilde} satisfies $g(x)\leq 4(2n+1)$ for every $x\in[n,n+1]$ and every $n\in\N$, by Remark \ref{rem:rad} it follows that
\[
\|\widetilde{w}\|_{L^2(\Rp,g\,dx)}^2\leq \|w\|_{L^2(\Q)}^2,\qquad\|\widetilde{w}'\|_{L^2(\Rp,g\,dx)}^2\leq\|w'\|_{L^2(\Q)}^2\,,
\]
where $\widetilde{w}:\Rp\to\R$ is the function associated to $w$ as in Definition \ref{u-rad}. Hence, choosing $v=\widetilde{w}$, we have $v\in X_\mu$, $v(0)=w(\overline{\vv})$ and 
\[
\widetilde{E}_{q,d}(v)\leq\frac12\|w'\|_{L^2(\Q)}^2-\frac dq|w(\overline{\vv})|^q\,,
\]
that together with \eqref{eq:wu} leads to \eqref{eq:v<u} and completes the proof of $\mu_q^*>0$.

It remains to show that $\lim_{q\to2^+}\mu_q^*=2$. For every $n\in\N$, $n\geq2$, let $f_{n}:\R^+\to \R$ be defined as
\begin{equation*}
f_{n}(x):=
\begin{cases}
\log n  & \text{if }x\in [0,1]\\
\log n -\log x  & \text{if }x\in [1,n]\\
0 & \text{if }x\in[n,+\infty)\,,
\end{cases}
\end{equation*}
and, for a fixed vertex $\ww\in V$, set $(w_n)_n\subset H^1(\Q)$ to be $w_n(x):=n^{-2}f_n(d(x,\ww))$ for every $x\in\Q$, where $d(x,\ww)$ denotes the distance in $\Q$ between $x$ and $\ww$.  By definition, $w_n$ is radial on $\Q$ with respect to $\ww$. Furthermore, recalling Remark \ref{rem:rad} and the definition of $f_n$, direct computations yields 
\[
\begin{split}
\frac{4}{n^{4}}\Bigg[\log^2 n &\,+\int_1^{+\infty}x(\log x -\log n )^2\,dx\Bigg]\leq\|w_n\|_{L^2(\Q)}^2\\
&\,\leq \frac{4}{n^{4}}\Bigg[\log^2 n +3\int_1^{+\infty}x(\log x-\log n )^2\,dx\Bigg]\,,
\end{split}
\] 
in turn giving, since $\displaystyle\int_1^{+\infty}x(\log x -\log n )^2\,dx=\frac{n^2}4+o(n^2)$ as $n\to+\infty$, 
\[
\mu_n:=\|w_n\|_{L^2(\Q)}^2= O(n^{-2})\qquad\text{as }n\to+\infty\,.
\] 
Furthermore, again by Remark \ref{rem:rad} we have that for every $n$
\[
\|w_n'\|_{L^2(\Q)}^2\leq 12\int_{1}^{+\infty} x|f_n'|^2\,dx=12n^{-4}\log n\,.
\]
Set then $\displaystyle q_n:=2+\frac1n$, so that 
\begin{equation*}
|w_n(\ww)|^{q_n}=n^{-(4+\frac{2}{n})}\log^{2+\frac{1}{n}}n, \quad\forall\,\,n\in \N,\,n\geq2\,,
\end{equation*}
and, as $n\to+\infty$, $\mu_n\to0$ and
\begin{equation*}
\begin{split}
\D_{q_n}\left(\mu_n\right)&\leq E_{q_n}(w_n)\leq \f{1}{2}\|w_n'\|_{L^2(\Q)}^2\left(1-\frac{2n}{2n+1}\frac{|w_n(\ww)|^{2+\frac{1}{n}}}{\|w_n'\|_{L^2(\Q)}^2}\right)\\
&\leq \f{1}{2}\|w_n'\|_{L^2(\Q)}^2\left(1-\frac{2n}{12(2n+1)}\frac{\log^{1+\frac{1}{n}} n}{n^{\frac{2}{n}}}\right)<0\,.
\end{split} 
\end{equation*}
This shows the existence of a sequence of exponents $\displaystyle q_n=2+\frac1n$ such that $\mu_{q_n}^*\to0$ as $n\to+\infty$. 

Observe now that, arguing exactly as in the proof of Lemma \ref{lem:q<q1} and making use of $\D_{q_n}(\mu_n)<0$, it follows that $\D_q(\mu_n)<0$ for every $q\in(2,q_n]$, that is $\mu_q^*\leq\mu_n$ for every $q\in(2,q_n]$.

Let then $(\widetilde{q}_n)_n\subset\R$ be any sequence such that $\widetilde{q}_n\to2^+$ as $n\to+\infty$. Since for every $n$ there exists $m=m(n)\in\N$ such that $\displaystyle\widetilde{q}_n\leq 2+\frac1m=q_m$ and $m(n)\to+\infty$ as $n\to+\infty$, we immediately have that $0<\mu_{\widetilde{q}_n}^*\leq \mu_{m}\to0$ as $n\to+\infty$. This shows that $\lim_{q\to2^+}\mu_q^*=0$ and concludes the proof.
\end{proof}

\begin{proof}[Proof of Proposition \ref{prop:rad}]
	By Theorem \ref{thm:Vfin1}, if $u\in H_\mu^1(\Q)$ is a ground state of $E_{q,V}$ with $V=\left\{\overline{\vv}\right\}$, then $E(u)\leq0$. This implies that (up to a change of sign) $u$ is a positive solution of \eqref{prob} with
	\[
	\lambda=\frac{|u(\overline{\vv})|^q-\|u'\|_{L^2(\Q)}^2}{\mu}=\left(1-\frac2q\right)\frac{|u(\overline{\vv})|^q}{\mu}-\frac{2 E(u)}{\mu}>0\,.
	\]
	The radiality of $u$ with respect to $\overline{\vv}$ follows directly  by Lemma \ref{lem:rad1}. Indeed, if $u$ were not radial, then Lemma \ref{lem:rad1} would ensure the existence of $w\in H^1(\Q)$ with $\|w\|_{L^2(\Q)}^2<\mu$ and $E(w)<E(u)\leq0$. Setting then $\displaystyle v:=\frac{\sqrt \mu}{ \|w\|_{L^2(\Q)}}\,w$ we would have $v\in H_\mu^1(\Q)$ and
	\[
	E(v)=\frac{\mu}{\|w\|_{L^2(\Q)}^2}\frac12\|w'\|_{L^2(\Q)}^2-\left(\frac{\mu}{\|w\|_{L^2(\Q)}^2}\right)^{\frac q2}\frac{|w(\overline{\vv})|^q}{q}<\frac \mu{\|w\|_{L^2(\Q)}^2} E(w)< E(w)<E(u)\,,
	\]
	which is impossible since $u$ is a ground state.
	
	To show that $u$ is decreasing with respect to $\overline{\vv}$ along the radial direction, let $\widetilde{u}\in H^1(\R^+)$ be the function associated to $u$ as in Definition \ref{u-rad}. Looking at $\Q$ as embedded in $\R^2$ with vertices on $\Z^2$ and $\overline{\vv}$ at the origin, $\widetilde{u}$ is e.g. the restriction of $u$ to the positive part of the $x$ axis. Since $u$ solves \eqref{prob} on $\Q$, is positive on $\Q$ and radial with respect to $\overline{\vv}$, it is immediate to see that $\widetilde{u}$ satisfies
	\begin{equation}
	\label{eq:utilde}
	\begin{cases}
	\widetilde{u}''=\lambda \widetilde{u} & \text{on }(n,n+1)\,,\forall n\in\N\\
	4\widetilde{u}_+'(0)=-\widetilde{u}^{q-1}(0) & \\
	\widetilde{u}_-'(n)=3\widetilde{u}_+'(n) & \forall n\in\N\setminus\{0\}\,.
	\end{cases}
	\end{equation}
	Note that the pointwise condition at $x=n$, for every $n\in\N\setminus\left\{0\right\}$, comes from the fact that, at the corresponding vertex of $\Q$, $u$ satisfies the homogeneous Kirchhoff condition and, by radiality, it agrees on three of the four edges emanating from that vertex.
	
	Set now $T:=\sup\{t\geq0\,:\, \widetilde{u}_+'\text{ is strictly negative on }[0,t)\}$. By \eqref{eq:utilde} and $\widetilde{u}>0$ on $\R^+$, we have $T>0$. To show that $\widetilde{u}$ is decreasing on $\R^+$, it is then enough to prove that $T=+\infty$. Assume by contradiction that this is not the case, i.e. $T<+\infty$. By definition of $T$ and \eqref{eq:utilde}, it then follows that $\widetilde{u}_-'(T)=\widetilde{u}_+'(T)=0$. Then, again by \eqref{eq:utilde} and $\widetilde{u}>0$ on $\R^+$, it follows that $\widetilde{u'}_+(x)>0$ on $(T,+\infty)$, which is impossible since $u\in H^1(\R^+)$. Hence, $T=+\infty$ and $\widetilde{u}$ is decreasing on $\R^+$. 
\end{proof}
   
\section{The grid $\Q$ with $V$ infinite: proof of Theorems \ref{thm:grid-Z-per}--\ref{thm:grid-Z2-per}--\ref{thm:noexQ}}
\label{sec:Vinf}

In this section, we take $V\subset\V_{\Q}$ such that $\#V=+\infty$ and we prove our main results in this setting, i.e. Theorems \ref{thm:grid-Z-per}--\ref{thm:grid-Z2-per}--\ref{thm:noexQ}.

\begin{remark}
\label{rem:fund-dom}
  Comparing Definitions \ref{def:I-Z-per}--\ref{def:I-Z2-per}  with that of periodic graph given in \cite[Definition 4.1.1]{BK}, it is easily seen that, if $V\subset\V_\Q$ is $\Z$-periodic according to Definition \ref{def:I-Z-per} and $\vec{v}\in\R^2\setminus\{0\}$ is the associated vector, then there exist $V_0\subset V$, with $\#V_0<+\infty$, and $\Q_0\subset\Q$, with $|\Q_0|=+\infty$, such that  
  \[
  V=\bigcup_{k\in\Z}\left(V_0+k\vec{v}\right),\quad\Q=\bigcup_{k\in\Z}\left(\Q_0+k\vec{v}\right)\quad\text{and}\quad \Q_0\cap V=V_0.
  \]
  Analogously, if $V\subseteq\V_\Q$ is $\Z^2$-periodic according to Definition \ref{def:I-Z2-per} and $\vec{v}_1,\vec{v}_2\in\R^2\setminus\{0\}$ are the associated vectors, then there exist $V_0\subset V$, with $\#V_0<+\infty$, and $\Q_0\subset\Q$, with $|\Q_0|<+\infty$, such that
  \[
  V=\bigcup_{k_1,k_2\in\Z}\left(V_0+k_1\vec{v}_1+k_2\vec{v}_2\right),\quad\Q=\bigcup_{k_1, k_2\in\Z}\left(\Q_0+k_1\vec{v}_1+k_2\vec{v}_2\right)\quad\text{and}\quad \Q_0\cap V=V_0.
  \]
  \end{remark}
Recall that, by Lemma \ref{lem:level} and \eqref{eq:lQ=0}, if $V\subseteq\V_\Q$ is $\Z$-periodic or $\Z^2$-periodic and $q\in(2,4)$, then
\[
-\infty<\D_{q,V}(\mu,\Q)\leq 0\,.
\]
The next lemma is the analogue of Lemmas \ref{lem:compact-crit}--\ref{lem:exZper} in the context of $\Z$-periodic and $\Z^2$-periodic sets $V$.
\begin{lemma}
\label{lem:compact-crit-grid-inf}
Let $\Q$ be the two--dimensional grid, $V\subseteq\V_{\Q}$ be $\Z$-periodic or $\Z^2$-periodic, and $q\in(2,4)$. If, for some $\mu>0$, there results $\D_{q,V}(\mu,\Q)<0$, then ground states of $E_{q,V}$ in $H_\mu^1(\Q)$ exist.
\end{lemma}
\begin{proof}
The proof is almost identical to that of Lemma \ref{lem:exZper}. 

If $V$ is $\Z$-periodic, let $(u_n)_n\subset H_\mu^1(\Q)$ be such that $E(u_n)\to\D(\mu)$ as $n\to+\infty$. Furthermore, exploiting the periodicity of $V$ there is no loss of generality in taking $u_n$ to satisfy $\sup_{\vv\in V}|u_n(\vv)|=\max_{\vv\in V_0}|u_n(\vv)|$ for every $n$, where $V_0$ is the set associated to $V$ as in the first part of Remark \ref{rem:fund-dom}. Then, arguing as in the proof of Lemma \ref{lem:exZper}, up to subsequences we have that $u_n\rightharpoonup u$ weakly in $H^1(\Q)$ and $u_n\to u$ in $L_{\text{\normalfont loc}}^\infty(\Q)$. The same computations in the final part of the proof of Lemma \ref{lem:compact-crit} guarantee that either $u\equiv0$ on $\Q$ or $u\in H_\mu^1(\Q)$ is the desired ground state. To rule out the first case, observe that if it were $u\equiv0$, then the argument in the proof of Lemma \ref{lem:exZper} and the fact that $\Q_0\cap V=V_0$, with $\#V_0<+\infty$ by construction, would imply $\displaystyle\sum_{\vv\in V}|u_n(\vv)|^q\to0$ as $n\to+\infty$, in turn yielding $\D(\mu)=\lim_{n\to+\infty}E(u_n)\geq0$, i.e. a contradiction. This proves the lemma when $V$ is $\Z$-periodic.

If $V$ is $\Z^2$-periodic, it is straightforward to see that we can take $(u_n)_n\subset H_\mu^1(\Q)$ such that $E(u_n)\to\D(\mu)$ as $n\to+\infty$ and $u_n$ attains its $L^\infty$ norm in $\Q_0$ for every $n$, where $\Q_0$ is the set associated to $V$ as in the second part of Remark \ref{rem:fund-dom}. Then the proof follows the same argument already discussed for $\Z$-periodic sets.
\end{proof}

\begin{proof}[Proof of Theorem \ref{thm:grid-Z-per}]
	We split the proof in two parts.
	
	\smallskip
	{\em Part 1: $q\in(2,3)$.} By Lemma \ref{lem:compact-crit-grid-inf}, to show that when $V$ is $\Z$-periodic and $q\in(2,3)$ ground states of $E_{q,V}$ in $H_\mu^1(\Q)$ exist for every $\mu$ it is enough to prove that $\D_{q,V}(\mu,\Q)<0$. 
	
	To this end, let $\vec{v}=(v_x,v_y)\in\R^2\setminus\{0\}$ be the vector associated to $V$ according to Definition \ref{def:I-Z-per} and suppose, without loss of generality, that the origin of $\R^2$ belongs to $V$. Moreover, let $R:=|v_x|+|v_y|$ and note that
	\[
	V\supset\bigcup_{k\in\Z}k\vec{v}\,,
	\]
	where with a slight abuse of notation we identify the vector $k\vec{v}$ with its final point in $\R^2$ (which is also a point of $V$ for every $k\in\Z$, by the periodicity of $V$).
	
 	Consider then, for every $\ep>0$, the function
	\begin{equation}
	\label{eq:varphi}
	\varphi_{\ep}(x,y):=k_{\ep}e^{-\ep(|x|+|y|)},\quad (x,y)\in \R^2,
	\end{equation}
	with
	\begin{equation}
	\label{eq:k}
	k_{\ep}:=\sqrt{\f{\ep\mu}{2}\f{1-e^{-2\ep}}{1+e^{-2\ep}}},
	\end{equation}
	and define $u_\ep\in H^1(\Q)$ as the restriction of $\varphi_\ep$ to the grid $\Q$.
By construction, $u_\ep\in H^1_\mu(\Q)$ and $\|u_\ep'\|_{L^2(\Q)}^2=\mu\ep^2$. Furthermore, as $\varepsilon\to0^+$,
	\[
	\sum_{\vv\in V}|u_\ep(\vv)|^q\geq \sum_{k\in\Z}|u_\ep(k\vec{v})|^q\geq  k_\ep^q\sum_{k=0}^{+\infty}e^{-\ep q R k}=\frac{k_\ep^q}{1-e^{-\ep qR}}= \frac{\mu^{q/2}}{2^{q/2}q R}\ep^{q-1}+o(\varepsilon^{q-1})\,,
	\]
	so that
	\[
	\D(\mu)\leq E(u_\ep)= \f{\mu\ep^2}{2}-\frac{1}{q}\sum_{\vv\in V}|u_\ep(\vv)|^q\leq \f{\mu\ep^2}{2}-\frac{\mu^{ q/2}}{2^{ q/2}q R}\ep^{q-1}+o(\varepsilon^{q-1})<0
	\]
	as soon as $\varepsilon$ is small enough and $q\in(2,3)$.
	
	\smallskip
	{\em Part 2: $q\in[3,4)$.} By Lemmas \ref{lem:mu-star}--\ref{lem:q-star}--\ref{lem:compact-crit-grid-inf}, it is enough to show that $\mu_3^*>0$. To this end, for every $\vv\in V$ denote by $e_\vv\in\E_\Q$ the horizontal edge for which $\vv$ is the left vertex. Set then $\displaystyle\G':=\bigcup_{\vv\in V}e_\vv$. Since $V$ is $\Z$-periodic, it follows that
	\[
	\min\left\{\sup_{j\in \Z}\#\left(\V_{\G'}\cap H_{j}\right), \sup_{j\in \Z}\#\left(\V_{\G'}\cap V_{j}\right) \right\}<+\infty\,.
	\]
	Indeed, by Definition \ref{def:I-Z-per}(ii), $V$ is fully contained in a strip of $\R^2$ parallel to the vector $\vec{v}$ and bounded in the direction $\vec{v}^\perp$. Since such  strip cannot contain simultaneously both a horizontal line $H_j$ and a vertical one $V_{j'}$, for any $j,j'\in\Z$, it follows that at least one between $\displaystyle \sup_{j\in\Z}\#\left(\V_{\G'}\cap H_{j}\right)$ and $\displaystyle \sup_{j\in \Z}\#\left(\V_{\G'}\cap V_{j}\right)$ is finite. 
	
	Hence, by Lemma \ref{lem:ineq-G'} with $q=3$ we have 
	\begin{equation}
		\label{eq:u3_1}
		\|u\|_{L^3(\G')}^3\leq K\sqrt{\mu}\|u'\|_{L^2(\Q)}^2\qquad\forall u\in H_\mu^1(\Q)\,.
	\end{equation}
	Furthermore, since the vertices of $V$ are by construction in one-to-one correspondence with the edges of $\G'$, for every $u\in H_\mu^1(\Q)$ we obtain
	\[
	\begin{split}
	\left|\sum_{\vv\in V}|u(\vv)|^{3}-\|u\|_{L^{3}(\G')}^{3}\right|&=\left|\sum_{\vv\in V}\int_{e_\vv}\left(|u(\vv)|^{3}-|u(y)|^{3}\right)\,dy\right|\leq 3\int_{\G'}|u|^{2}|u'|\,dy\\
	&\leq 3 \|u'\|_{L^{2}(\Q)}\|u\|_{L^4(\Q)}^{2}\leq 3\sqrt{M_4}\sqrt{\mu}\|u'\|_{L^{2}(\Q)}^{2}\,,
	\end{split}
	\]
	where we used the H\"older inequality and \eqref{GN-2d} with $p=4$. Coupling with \eqref{eq:u3_1} and plugging into the definition of $E_3$ leads to
	\begin{equation*}
	E_{3}(u)=\f{1}{2}\|u'\|_{L^2(\Q)}^2-\f{1}{3}\sum_{\vv\in V}|u(\vv)|^3\geq \f{1}{2}\|u'\|_{L^{2}(\Q)}^{2}\left(1-\frac{K+3\sqrt{M_4}}{3}\sqrt{\mu}\right)>0
	\end{equation*}
	for every $u\in H_\mu^1(\Q)$ as soon as $\mu$ is sufficiently small, i.e. $\mu_3^*>0$ and we conclude.
\end{proof}

\begin{proof}[Proof of Theorem \ref{thm:grid-Z2-per}]
	By Lemma \ref{lem:compact-crit-grid-inf}, if $V$ is a $\Z^2$-periodic subset of $\V_\Q$ according to Definition \ref{def:I-Z2-per}, to prove Theorem \ref{thm:grid-Z2-per} it is enough to show that $\D_{q,V}(\mu,\Q)<0$ for every $q\in(2,4)$ and $\mu>0$.
	
	To this end, let $\vec{v}_1,\vec{v}_2\in\R^2\setminus\{0\}$ be the two linearly independent vectors associated to $V$ as in Definition \ref{def:I-Z2-per}, and assume with no loss of generality that the origin of $\R^2$ belongs to $V$. Hence, by the periodicity of $V$, the vertices $k_1\vec{v}_1+k_2\vec{v}_2$ belong to $V$, for every $k_1,k_2\in\Z$ (here and in the following we identify again vectors in $\R^2$ with the corresponding final points). 
	
	Let $\gamma\subset\Q$ be a path of minimum length in $\Q$ connecting the origin (that by assumption is a vertex in $V$) to the vertex $\vec{v}_{1}$ (that is again a vertex in $V$), and define the subgraph $\Gamma_{0}$ of $\Q$ as
	\begin{equation*}
	\Gamma_{0}:=\bigcup_{k_{1}\in\Z}\left(\gamma+k_{1}\vec{v}_{1}\right).
	\end{equation*}
	By construction, $\Gamma_{0}$ is connected, has infinite length and contains all the vertices of $V$ of the form $k_1\vec{v}_1$, with $k_1\in \Z$. Denote also by $R_{1}:=|\gamma|=|v_{1,x}|+|v_{1,y}|$ the length of $\gamma$, namely,  by construction, the distance in $\Gamma_{0}$ between two consecutive vertices $k_{1}\vec{v}_{1}$ and $(k_{1}+1)\vec{v}_{1}$, for every $k_1$. 
	
	Observe that, since the length of $\gamma$ is finite and $\vec{v}_{1}$ and $\vec{v}_{2}$ are linearly independent, there exists $\widetilde{k}\in \N$ such that $\widetilde{k}$ is the smallest natural number for which $\gamma+\widetilde{k}\vec{v_{2}}$ does not intersect $\gamma$. For every $k_{2}\in\N\setminus\{0\}$, define then the subgraph $\Gamma_{k_{2}}$ of $\Q$ as
	\begin{equation*}
	\Gamma_{k_{2}}:=\Gamma_{0}+k_{2}\widetilde{k}\vec{v_{2}}.
	\end{equation*}  
	By construction, $\Gamma_{k_{2}}\cap\Gamma_{k_{2}'}=\emptyset$ for every $k_2, k_2'\in \N$, $k_2\neq k_2'$. Moreover, for every $k_{2}\in \N$,  $\Gamma_{k_{2}}$ is connected, has infinite length and contains all the vertices of $V$ in the form $k_1\vec{v}_1+k_2\widetilde{k}\vec{v}_2$, with $k_1\in\Z$. Furthermore,  the distance in $\Gamma_{k_{2}}$ between the vertices $k_{1}\vec{v_{1}}+k_{2}\widetilde{k}\vec{v_{2}}$ and $(k_{1}+1)\vec{v_{1}}+k_{2}\widetilde{k}\vec{v_{2}}$ is again $R_{1}$, for every $k_1$, whereas the distance between $\Gamma_{k_2}$ and $\Gamma_0$ in $\Q$ is $k_2R_2$, with $R_2:=\widetilde{k}(|\vec{v}_{2,x}|+|\vec{v}_{2,y}|)$.
	
	Fix now $\mu>0$ and, for every $\varepsilon>0$, let $\varphi_\varepsilon,k_\varepsilon$ be as \eqref{eq:varphi}, \eqref{eq:k} and, as in the proof of Theorem \ref{thm:grid-Z-per}, set $u_\varepsilon\in H_\mu^1(\Q)$ to be the restriction of $\varphi_\varepsilon$ to the grid $\Q$. As above, $u_{\ep}\in H^{1}_{\mu}(\Q)$ and $\|u_{\ep}'\|_{L^{2}(\Q)}^{2}=\ep^{2}\mu$ for every $\ep>0$. Moreover, as $\varepsilon\to0^+$ (recalling also the definition \eqref{eq:k} of $k_\varepsilon$)
	\begin{equation*}
	\begin{split}
	\sum_{\vv\in V}|u_\ep(\vv)|^{q}&\geq \sum_{k_2=0}^{+\infty}\sum_{\vv\in V\cap \Gamma_{k_2}}|u_\varepsilon(\vv)|^q\geq k_{\ep}^{q}\sum_{k_{2}=0}^{+\infty}\sum_{k_1=0}^{+\infty}e^{-\ep q\left(|k_1\vec{v}_{1,x}+k_2\vec{v}_{2,x}|+|k_1\vec{v}_{1,y}+k_2\vec{v}_{2,y}|\right)}\\
	&\geq k_{\ep}^{q}\sum_{k_1=0}^{+\infty}e^{-\varepsilon q k_1 R_1}\sum_{k_2=0}^{+\infty}e^{-\varepsilon q k_2 R_2}=\frac{\mu^{ q/2}}{2^{ q/2}q^2 R_1 R_2}\varepsilon^{q-2}+o(\varepsilon^{q-2})\,,
	\end{split}
	\end{equation*}
	so that for every $q\in(2,4)$
	\begin{equation*}
	\D(\mu)\leq E(u_{\ep})\leq \f{\mu}{2}\ep^{2}-\frac{\mu^{ q/2}}{2^{ q/2}q^3 R_1 R_2}\varepsilon^{q-2}+o(\varepsilon^{q-2})<0
	\end{equation*}
	taking sufficiently small $\varepsilon$.
\end{proof}

\begin{proof}[Proof of Theorem \ref{thm:noexQ}]
	
	To prove Theorem \ref{thm:noexQ}(i) it is enough to exhibit an example of a set $V$ that is not $\Z$-periodic but is contained in a strip of $\R^2$, bounded in one direction and unbounded in the orthogonal one, for which ground states never exist, independently of $q$ and $\mu$. To do this, it is enough to repeat verbatim the argument of Step 3 in the proof of Theorem \ref{thm:Z-per}. We identify as usual $\Q$ with the corresponding subset of $\R^2$ with vertices on $\Z^2$ and we denote by $\vv_{ij}$ the vertex with coordinates $(i,j)$ for every $i,j\in\Z$. Let $a_n = n(n+1)$, $n\in \N$, as in the proof of  Theorem \ref{thm:Z-per} and define 
	\[
	V =\{ v_{i0} \mid i \ne  a_n \; \forall n\in \N\}
	\]
By construction, $V$ is contained in the horizontal strip $\R\times[-1,1]$ of $\R^2$, so that it satisfies Definition \ref{def:I-Z-per}(ii), but it is not $\Z$-periodic since it does not satisfy Definition \ref{def:I-Z-per}(i). Then, arguing exactly as in Step 3 of the proof of Theorem \ref{thm:Z-per} we obtain again that for every $u\in H_\mu^1(\Q)$, $u>0$ on $\Q$, there exists $w\in H_\mu^1(\Q)$ (a translation of $u$ along  the $x$ axis) for which $E_{q,V}(w,\Q)<E_{q,V}(u,\Q)$. As this rules out existence of ground states in $H_\mu^1(\Q)$, it concludes the proof of Theorem \ref{thm:noexQ}(i).
	
	To prove Theorem \ref{thm:noexQ}(ii) it is enough to copy the set used above on every horizontal line, namely to define
	\[
	V =\{ v_{ij} \mid i \ne  a_n \; \forall n\in \N, \; j\in \Z \}.
	\]
By construction, $V$ satisfies Definition \ref{def:I-Z-per}(i), since it is periodic in the vertical direction, but it does not satisfy neither Definition \ref{def:I-Z-per}(ii) nor Definition \ref{def:I-Z2-per}. Arguing again as in Step 3 of the proof of Theorem \ref{thm:Z-per}, it is easy to show that ground states of $E_{q,V}$ in $H_\mu^1(\Q)$ do not exist for any $q\in(2,4)$ and $\mu>0$, thus completing the proof.
\end{proof}

\section*{Acknowledgements}
\noindent The work has been partially supported by the INdAM GNAMPA project 2023 ``Modelli nonlineari in presenza di interazioni puntuali'' and by the PRIN 2022 project E53D23005450006. The authors wish to thank Paolo Tilli for fruitful discussions during the preparation of this work.

\appendix

\section{Linear problems}
\label{app:linear}
In this appendix we briefly collect some results on linear problems that have been used in the previous analysis and that can be compared with the main results of the paper in the nonlinear setting. 

The content of the next lemma concerning $\lambda(\G)$ (as defined in \eqref{lambdaG}) is standard and well-known but, since it has been widely used in the paper, we report here a short proof for the sake of completeness.
\begin{lemma}
		\label{lem:bot-spec-0}
		Let $\G\in \mathbf{G}$ be either a graph with at least one half-line, a $\Z$-periodic graph or the two-dimensional grid $\Q$. Then $\la(\G)=0$.
	\end{lemma}
	\begin{proof}
		If $\G$ has at least one half-line the result is obvious, since in this case $0\leq \lambda(\G)\leq \lambda(\R)=0$.
		
		If $\G$ is a $\mathbb{Z}$-periodic graph with periodicity cell $\K$ according to the definition given in Section \ref{sec:prel}, for every $n\in\N$ let $u_n:\G\to\R$ be such that $u_n\equiv 1$ on $\K_i$ for every $0\leq|i|\leq n$, $u_n$ is linearly decreasing from $1$ to $0$ on every edge starting at a vertex of $\K_n$ and ending at a vertex of $\K_{n+1}$ and on every edge starting at a vertex of $\K_{-n}$ and ending at a vertex of $\K_{-n-1}$, and $u_n\equiv0$ elsewhere on $\G$. Since, by periodicity, the number of edges between $\K_{n}$ and $\K_{n+1}$ and that between $\K_{-n}$ and $\K_{-n-1}$ is finite and independent of $n$, as $n\to+\infty$ we have
		\[
		\|u_n\|_{L^2(\G)}^2 = (2n+1)|\K|+O(1) ,\qquad\|u_n'\|_{L^2(\G)}^2=C\,,
		\] 		
		for a suitable constant $C>0$ independent of $n$. Hence, 
		\[
		0\leq \lambda(\G)\leq\lim_{n\to+\infty}\frac{\|u_n'\|_{L^2(\G)}^2}{\|u_n\|_{L^2(\G)}^2}=0
		\]
		which proves the claim in the case of $\Z$-periodic graphs.
		
		Finally, if $\G=\Q$ is the two--dimensional square grid, we think of it as a subset of $\R^2$ with vertices in $\Z^2$. For every $n\in\N$, let then $Q_n$ be the intersection of $\Q$ with the square $[0,n]^2$ in $\R^2$, and define $u_n:\Q\to\R$ so that $u_n\equiv1$ on $Q_n$, $u_n\equiv0$ on $\Q\setminus Q_{n+1}$, and $u_n$ decreases linearly from $1$ to $0$ on every edge of $Q_{n+1}\setminus Q_n$. Direct computations show that, as $n\to+\infty$,
		\[
		\|u_n\|_{L^2(\Q)}^2=2n(n+1)+O(n),\qquad\|u_n'\|_{L^2(\Q)}^2=4(n+1)\,,
		\]
		which is again enough to conclude that $\lambda(\Q)=0$.
	\end{proof}
	\begin{remark}
		Even though the statement of Lemma \ref{lem:bot-spec-0} is limited to the families of graphs covered in this paper, it is evident that the same arguments show that $\lambda(\G)=0$ on every infinite periodic graph satisfying \cite[Definition 4.1.1]{BK}.
	\end{remark}
	
	Let us now consider the following minimization problem. Fix a vertex $\vv\in\V_{\G}$ and a parameter $\alpha>0$ and set
	\begin{equation}
	\label{inf-lin-delta}
	\la_{\alpha,\vv}(\G):=\inf_{w\in H^1(\G)\setminus\{0\}}\f{\|w'\|_{L^2(\G)}^2-\alpha|w(\vv)|^2}{\|w\|_{L^2(\G)}^2}\,.
	\end{equation}
	It is straightforward to see that, if $\lambda_{\alpha,\vv}(\G)$ is attained by some function $u\in H^1(\G)$, then for every $\mu>0$ there exists $v\in H_\mu^1(\G)$ satisfying
	\begin{equation}
	\label{eq:linpro}
	\begin{cases}
	v_{\ee}''=-\la_{\alpha,\vv} v_{\ee} & \quad \text{on every } \ee \in \E_{\G},\\
	\sum_{\ee\succ \vv}v_{e}'(\vv)=-\alpha v(\vv)& \\
	\sum_{\ee\succ \ww}v_{e}'(\ww)=0& \quad \forall\,\ww\in \V_{\G}\setminus \{\vv\},
	\end{cases}
	\end{equation}
	i.e. the linear counterpart of \eqref{prob}. Furthermore, arguing as in the previous sections, it is easy to see that $\lambda_{\alpha,\vv}(\G)$ is attained if it is strictly smaller than $\lambda(\G)$. The next lemma, together with Lemma \ref{lem:bot-spec-0} above, shows that when $\G=\Q$ the linear problem \eqref{eq:linpro} has always a solution, in sharp contrast with the nonlinear case as discussed in Theorem \ref{thm:Vfin1}.
	\begin{lemma}
		\label{lem:bot-spec<0}
		For every $\vv\in\V_{\Q}$ and $\alpha>0$, there results $-\infty<\la_{\alpha,\vv}(\Q)<0$.
	\end{lemma}
	\begin{proof}
		By homogeneity,  $\lambda_{\alpha,\vv}(\Q)=\inf_{w\in H_1^1(\Q)}(\|w'\|_{L^2(\Q)}^2-\alpha|w(\vv)|^2)$. By \eqref{GN-inf}, for every $w\in H_1^1(\Q)$ one has
		\[
		\|w'\|_{L^2(\Q)}^2-\alpha|w(\vv)|^2\geq \|w'\|_{L^2(\Q)}^2-\alpha C_\infty\|w'\|_{L^2(\Q)}\,,
		\]
		so that $\lambda_{\alpha,\vv}(\Q)>-\infty$.
		
		For every $n\in \N\setminus\{0\}$, consider now the radial function $u_{n}\in H^{1}(\Q)$ given by $u_{n}(x)=f_{n}(d(x,\vv))$, with $f_{n}:[0,+\infty)\to \R$ defined as
		\begin{equation*}
		f_{n}(x):=
		\begin{cases}
		\log n& \text{if }x\in [0,1]\\
		\log n -\log x &\text{if } x\in [1,n]\\
		0&\text{if } x\in[n,+\infty)\,.
		\end{cases}
		\end{equation*}
		It is straightforward to check, recalling Remark \ref{rem:rad}, that
		\begin{equation*}
		\|u_{n}'\|_{L^{2}(\Q)}^{2}\leq 12\int_{1}^{n} x|f_{n}'(x)|^{2}\,dx=12 \log n,
		\end{equation*}
		so that combining with $|u_{n}(0)|^{2}=\log^{2}n$ we obtain
		\begin{equation*}
		\lambda_{\alpha,\vv}(\Q)\leq\f{\|u_{n}'\|_{L^{2}(\Q)}^{2}-\alpha|u_{n}(0)|^{2}}{\|u_{n}\|_{L^{2}(\Q)}^{2}}\leq \f{12\log n -\alpha\log^{2} n }{\|u_{n}\|_{L^{2}(\Q)}^{2}}.
		\end{equation*}
		Since $12\log n-\alpha\log^{2} n<0$ when $n$ is large enough, we conclude.
	\end{proof}

\end{document}